\documentclass[11pt,a4paper,reqno]{amsart}

\usepackage{amsmath,amsfonts,amssymb,amsthm,mathrsfs}
\usepackage[utf8]{inputenc}
\usepackage{lmodern}
\usepackage{latexsym}
\usepackage[T1]{fontenc}
\usepackage{cancel}
\usepackage{microtype}
\usepackage{caption}
\usepackage{MnSymbol}
\usepackage{xcolor}
\usepackage{enumerate}
\usepackage[normalem]{ulem}
\usepackage{bbm}
\usepackage{geometry}

\newcommand{\BB}{{\cal B}}
\newcommand{\DD}{{\cal D}}
\newcommand{\EE}{{\cal E}}
\newcommand{\FF}{{\cal F}}

\newcommand{\TT}{{\cal T}}

\newcommand{\BM}{{\mathbb M}}

\newcommand{\BR}{{\mathbb R}}

\newcommand{\BE}{{\mathbb E}}
\newcommand{\BF}{{\mathbb F}}

\newcommand{\esssup}{\mathop{\mathrm{ess\,sup}}}
\newcommand{\fch}{{\mathbf{1}}}
\newcommand{\cal}{\mathcal}

\newtheorem{theorem}{\bf Theorem}[section]
\newtheorem{proposition}[theorem]{\bf Proposition}
\newtheorem{lemma}[theorem]{\bf Lemma}%[subsection]
\newtheorem{corollary}[theorem]{\bf Corollary}

\theoremstyle{definition}
\newtheorem{definition}[theorem]{Definition}
\newtheorem{example}[theorem]{\bf Example}

\newtheorem{remark}[theorem]{Remark}

\numberwithin{equation}{section}

\renewcommand{\thefootnote}{{}}

\usepackage{graphics,tikz}

\usepackage{amsaddr}
\usepackage{lipsum}
\makeatletter
\g@addto@macro{\endabstract}{\@setabstract}
\newcommand{\authorfootnotes}{\renewcommand\thefootnote{\@fnsymbol\c@footnote}}%
\begin{document}

\title[Nonlinear stopping problems]{Long-time asymptotic behaviour of the value function
in nonlinear  stopping problems}

\maketitle

\begin{center}

\normalsize
\authorfootnotes
TOMASZ KLIMSIAK\footnote{e-mail: {\tt tomas@mat.umk.pl}}
\textsuperscript{1,2},
\quad ANDRZEJ ROZKOSZ \footnote{e-mail: {\tt rozkosz@mat.umk.pl} (corresponding author)}\textsuperscript{2}
\par \bigskip

\textsuperscript{1} {\small Institute of Mathematics, Polish Academy of Sciences,\\
\'{S}niadeckich 8,   00-656 Warsaw, Poland} \par \medskip

\textsuperscript{2} {\small Faculty of
Mathematics and Computer Science, Nicolaus Copernicus University,\\
Chopina 12/18, 87-100 Toru\'n, Poland }\par

\end{center}

%\date{}
\noindent{\bf Abstract}
We provide general conditions ensuring that the value functions of
some nonlinear stopping problems with finite horizon converge to
the value functions of the corresponding problems  with infinite
horizon. Our result can
be formulated as result on stability, with respect to time
horizon, of nonlinear $f$-expectations. We also study the rate of convergence.
Many examples are given to illustrate our results. They include the analysis
of time asymptotics of the fair prices
of American options  in a multidimensional exponential L\'evy model.

\footnotetext{This work was supported by Polish National Science Centre under grant no. 2016/23/B/ST1/01543.}

\section{Introduction}

Let $\BM=\{(X,P_x),x\in E\}$ be a Borel right Markov  process with
state space $E$ and life time $\zeta$, and $D$ be an open subset
of $E$. For  $T>0$ and measurable functions
$g:E\times\BR\rightarrow\BR$, $h,\varphi:E\rightarrow\BR$ and
$\psi:E\setminus D\to \BR$ we consider the value functions
\begin{align}
\label{eq1.01}
V_T(s,x)=\sup_{\sigma\le T_s\wedge\tau_D}\EE^{x,f}_{0,\sigma}\big[h(X_{\sigma})
&\fch_{\{\sigma<T_s\wedge\tau_D\}}+\psi(X_{\tau_D})\fch_{\{\sigma=\tau_D,
\tau_D\le T_s\}}\nonumber\\
& +\varphi(X_{T_s})\fch_{\{\sigma=T_s,  \tau_D>T_s\}}\big], \quad(s,x)\in[0,T]\times D,
\end{align}
and
\begin{equation}
\label{eq1.02}
V(x)=\sup_{\sigma\le \tau_D}\EE^{x,f}_{0,\sigma}\big[h(X_{\sigma})\fch_{\{\sigma<\tau_D\}}
+\psi(X_{\tau_D})\fch_{\{\sigma=\tau_D\}}\big],\quad x\in D.
\end{equation}
Here
\[
f(t,y)=g(X_t,y),\quad t\ge0,\,y\in\BR,
\]
$\EE^{x,f}$ is the nonlinear $f$-expectation (defined under the
measure $P_x$) introduced by Peng \cite{Pe1} (see Section \ref{sec2}),  $\sigma$ are stopping times
with respect to the minimum admissible filtration $\BF=(\FF_t)_{t\ge0}$ generated by $\BM$, and
\[
\tau_D=\inf\{t>0: X_t\notin D\},\qquad T_s=T-s.
\]
Perhaps it is appropriate at this point  to note  that under  some
assumptions, $V_T$ and  $V$ are solutions of some "usual" nonlinear stopping problems, i.e. can
be defined without recourse to the notion of nonlinear
expectation. Specifically, from our results it follows that under natural assumptions on the data $g,h,\psi$ and $\varphi$ the function  $V_T$ is a solution of the nonlinear
equation
\begin{align}
\label{eq1.2} V_T(s,x)&=\sup_{0\le\sigma\le T_s\wedge\tau_D}
\BE_x\Big[\int^{\sigma}_0g(X_t,V_T(s+t,X_t))\,dt
+h(X_{\sigma})\fch_{\{\sigma<T_s\wedge\tau_D\}}\nonumber\\
&\qquad\qquad\qquad\qquad
+\psi(X_{\tau_D})\fch_{\{\sigma=\tau_D,\tau_D\le T_s\}}
+\varphi(X_{T_s})\fch_{\{\sigma=T_s,\tau_D>T_s\}}\Big],
\end{align}
and $V$ is a solution of the equation
\begin{equation}
\label{eq1.3} V(x)=\sup_{0\le\sigma\le
\tau_D}\BE_x\Big[\int^{\sigma}_0g(X_t,V(X_t))\,dt
+h(X_{\sigma})\fch_{\{\sigma<\tau_D\}}
+\psi(X_{\tau_D})\fch_{\{\sigma=\tau_D\}}\Big],
\end{equation}
where $\BE_x$ denotes the expectation with respect to $P_x$ (for
details and a generalization see Section \ref{sec4}). We stress, however, that in some of our results the generator $f$   is  merely
continuous and nonincreasing with respect to the $y$-variable.  Therefore, in general, the integrals in \eqref{eq1.2}, \eqref{eq1.3}
%$\BE_x\int_0^{T_s\wedge\tau_D}g(X_r,V_T(s+r,X_r))\,dr$ and  $\BF\int_0^{\tau_D}g(X_r,V(r,X_r))\,dr$
involving $g$ need not be well defined. Nevertheless, one can still define and study $V_T$, $V$ defined by (\ref{eq1.01}), (\ref{eq1.02}).

In this paper, we give general
conditions on the data $g,h,\varphi,\psi$ guaranteeing that
\begin{equation}
\label{eq1.1} \lim_{T_s\rightarrow\infty} V_T(s,x)=V(x),\quad x\in D.
\end{equation}
We also provide some estimates on the rate of convergence.
In many
cases  it is very important to have some information about the
dynamic of the value functions, i.e. about the  processes
$t\mapsto V_T(s+t,\cdot,X_t)$ and $t\mapsto V(X_t)$. The second
main result of the paper is the dynamic version of (\ref{eq1.1}).
We first prove that if $D$ is Dirichlet regular (i.e.
$P_x(\tau_D>0)=0$ for $x\in\partial D$), then for every stopping
time $\alpha\le \tau_D\wedge T_s$,
\begin{align*}
V_T(s+\alpha,X_\alpha)&=\esssup_{\alpha\le\sigma\le
T_s\wedge\tau_D}\EE^{x,f}_{\alpha,\sigma} \big[h(X_{\sigma})
\fch_{\{\sigma<T_s\wedge\tau_D\}}\\
&\qquad\qquad+\psi(X_{\tau_D})\fch_{\{\sigma=\tau_D, \tau_D\le
T_s\}} +\varphi(X_{T_s})\fch_{\{\sigma=T_s,  \tau_D> T_s\}}\big],
\end{align*}
and  for every stopping time $\alpha\le\tau_D$ we have
\[
V(X_\alpha)=\esssup_{\alpha\le\sigma\le\tau_D}
\EE^{x,f}_{\alpha,\sigma}
\big[h(X_{\sigma})\fch_{\{\sigma<\tau_D\}}
+\psi(X_{\tau_D})\fch_{\{\sigma=\tau_D\}}\big],\quad x\in D.
\]
We show that these two formulas  together with the integrability condition
\[
\mathbb E_x\int_0^{T_s\wedge\tau_D}|g(X_t,V_T(s+t,X_t))|\,dt
+\BE_x\int_0^{\tau_D}|g(X_t,V(t,X_t))|\,dr<\infty
\]
imply  \eqref{eq1.2}, \eqref{eq1.3}.
Then we prove that
\begin{align}
\label{eq1.4} \nonumber V^{*}_{T_s}(x)&:=\sup_{\alpha\le \tau_D\wedge T_s}
\BE_x|V_T(s+\alpha,X_\alpha)-V(X_\alpha)| \le
\BE_x|\gamma-\varphi|(X_{T_s})|\fch_{\{\tau_D\ge T_s\}}\\
&\quad+ \BE_x\int_{T_s\wedge\tau_D}^{\tau_D}|g(X_t,\gamma(X_t))|\,dt
+\sup_{T_{s}\wedge \tau_D\le \tau\le\tau_D}
\BE_x|\gamma(X_{\tau})-h(X_{\tau})|,
\end{align}
where  $\gamma(X_t)=\BE_x(\psi(X_{\tau_D})|\FF_t)$, and that
the right-hand side of (\ref{eq1.4}) converges to zero as
$T_s\to\infty$. From (\ref{eq1.4}) and  the elementary inequality
\begin{equation}
\label{eq1.5}
\BE_x\sup_{t\le T_s\wedge\tau_D} |V_T(s+t,X_t)-V(X_t)|^q\le\frac{1}{1-q}\big(V^*_{T_s}(x)\big)^q,\quad q\in (0,1),
\end{equation}
we get uniform (in $t$) convergence of the value processes.

The key idea of our proofs of (\ref{eq1.1}) and
\eqref{eq1.4} is to look at the processes $V(X)$ and  $V_T(s+\cdot,X)$
as solutions to certain reflected backward stochastic differential
equations (RBSDEs). In the non-dynamic case this link is
immediate. Some work is required to prove it in the dynamic case.
After establishing these links, we  get the desired results by
applying  new stability results for RBSDEs obtained recently in
\cite{K:SPA3} (see  Section \ref{sec3}).

In the second part of the paper (Sections \ref{sec5} and
\ref{sec6}), we study the rate of convergence in (\ref{eq1.1}). To
this end, in Section \ref{sec5} we describe several methods of
estimating the  right-hand side of (\ref{eq1.4}). In general,
these methods are incomparable. Its applicability depends on the
underlying process $\BM$ (L\'evy type process,  (intrinsic) ultracontractive process or symmetric Hunt process related to a symmetric Dirichlet form) and assumptions on the data
$\varphi,\psi,h$ (bounded or in $L^q$ with some $q\ge1$) and $D$
(bounded, unbounded or of finite reference measure, with regular boundary or not).
To illustrate the utility of these methods, let us mention here
that in  Section \ref{sec5}, among other things,  we get several results of type
 \[
 V^*_{T_s}(x)\sim e^{-T_s\lambda(x)},\quad x\in D,\quad\mbox{as}\quad T_s\to \infty.
 \]
Depending on the structure of the problem, the function  $\lambda(\cdot)$ may be   positive  and vanishing near  the boundary of $D$,  may be constant and equal to a  number  $\lambda>0$  or be equal to the principal eigenvalue $\lambda_1$ associated with the semigroup $(P^D_t)$ of the process $\BM$ killed upon leaving $D$.

The problem of controlling the rate of decay of the  right-hand side of
(\ref{eq1.4}) is by no means obvious.
Let  $R^D$ denote the   potential operator
associated with the process $\BM$ killed upon leaving $D$. We prove that for general Markov process $\BM$ we have
\[
V^{*}_{T_s}(x)\le P^D_{T_s}w(x)+P^D_{T_s}\hat h_{\infty}(x),\quad x\in D,\,0\le s \le T,
\]
where
\[
w(x)=|\gamma-\varphi|(x)+R^D(|g(\cdot,\gamma)|)(x),\qquad
\hat h_{\infty}(x)=\sup_{\tau\le\tau_D}\hat h(x,\tau).
\]
and
\[
%\hat h_{\infty}(x)=\sup_{\tau\le\tau_D}\hat h(x,\tau), \qquad
\hat h(x,\tau)=\BE_x|\gamma(X_\tau)-h(X_\tau)|.
\]
One easily checks that $P^D_{T_s}w(x)$ is equal to the sum of the first two terms on the right-hand side of \eqref{eq1.4}.
Thus, the study of the rate of convergence of
these  two terms  reduces to the study of the rate of
decay, as $t\rightarrow\infty$, of the semigroup $(P^D_t)$.
%We show by examples, that depending on the situation one can exponential decay
%of $P^D_{T _s}w(x)$
%for every $x\in D$) or the $L^2$-norm
For this one can use some known results from the semigroup
theory and probabilistic potential theory. In Section \ref{sec5}, we indicate some of them. The third
term on the right-hand side of (\ref{eq1.4}) causes  specific problems due to the compound term $\hat h_\infty$. We show how to deal with this term
in some  typical situations,  for instance, when $h(X)$ is a submartingale or supermartingale under $P_x$ for $x\in E$ or $\gamma,h\in C^2_b(\BR^d)$ and the generator of $\BM$ is a L\'evy-type operator. Another interesting situation we consider is when $\BM$ is associated with a symmetric regular Dirichlet form and $\gamma,h$ belong to its extended domain. In this case, in general, the process $h(X)$ need not be a semimartingale.

In Section \ref{sec6}, we provide a detailed analysis of (\ref{eq1.1}) in a stopping problem arising in the study of American options. We consider dividend paying options in an exponential L\'evy model. We show that if the payoff function is continuous and satisfies the linear growth condition, then  under natural assumptions on the model the fair price of the option with maturity $T$ converges as $T\rightarrow\infty$ to the fair price of the corresponding perpetual American option, and we provide the rate of convergence.

In the present paper, we deal exclusively with stopping problems and RBSDEs with one (lower) barrier. RBSDEs with two continuous barriers satisfying Mokobodzki's separability condition and some integrability conditions were introduced in \cite{CK} in case the underlying filtration is Brownian. In \cite{CK} also a connection of such equations with a pair of some coupled optimal stopping problems (so-called Dynkin's games) is given. At present, the theory initiated in \cite{CK} is quite well developed. The existence and uniqueness of  solutions of RBSDEs with possibly infinite terminal time is known to hold under fairly general assumptions on the data and for general filtration (see \cite{K:SPA3,KR:EJP}
and the references therein). It is also known that under fairly general assumptions the value functions in Dynkin's games can be represented by solutions of RBSDEs (see \cite{K:SPA3,KR:EJP}). It would be interesting to use these results to give, in the case of two barriers, some asymptotic results similar to those given in the present paper.
%To our knowledge, such results are not known even in case of Brownian filtration
%and regular barriers satisfying Mokobodzki's condition(?) (ostatnie zdanie prawdziwe? %opuscic?)

\section{RBSDEs and nonlinear optimal stopping problem}
\label{sec2}

For the sake of completeness, in this section we recall some known results on the existence and uniqueness of solutions to backward stochastic differential equations (BSDEs) and reflected BSDEs (RBSDEs) with one barrier.

In what follows $(\Omega,\FF,P)$ is a complete probability space
and $\BF=(\FF_t)_{t\ge0}$ is a right-continuous  filtration
satisfying the usual conditions. We denote by $\TT$ the set of all $\BF$-stopping times, and for given $\BF$-stopping times $\tau,\sigma$ such that $\tau\le\sigma$ we denote by $\TT^{\sigma}_{\tau}$ the set of all $\BF$-stopping times $\alpha$ such that $\tau\le\alpha\le\sigma$. To simplify notation, we write $\TT^{\sigma}$ for $\FF^{\sigma}_0$ and $\TT_{\tau}$ for $\TT^{\infty}_{\tau}$.

Below we assume as given an $\BF$-stopping time $\vartheta$ (not
necessarily finite),  an $\FF_{\vartheta}$-measurable random
variable $\xi$, an $\BF$-adapted c\`adl\`ag  process $L$ of class
(D) and a function $f:\Omega\times [0,\infty)\times\BR\rightarrow
\BR$ such that $f$ is $\BF$-adapted with respect to
$(\omega,t)\in\Omega\times [0,\infty)$ for any fixed $y\in\BR$. As
usual, in the sequel in our notation we omit the dependence of $f$
on $\omega\in\Omega$.

Recall that a c\`adl\`ag $\BF$-adapted process $Y$ is said to be of class  (D) under
the measure $P$ if the collection of random variables
$\{Y_{\tau}:\tau\in\TT,\ \tau<\infty\}$ is uniformly integrable under $P$.

\begin{definition}
We say that a pair $(Y,M)$ of $\mathbb F$-adapted c\`adl\`ag processes  is a
solution, on the interval $[0,\vartheta]$, of the  BSDE  with
terminal condition $\xi$ and coefficient $f$
(BSDE${}^{\vartheta}(\xi,f)$ for short) if $Y$ is a  process of
class (D),  $M$ is a local martingale such that $M_0=0$,  $P$-a.s., and the following
conditions are satisfied $P$-a.s.:
\begin{enumerate}[(a)]
\item $\int^{a\wedge\vartheta}_0|f(t,Y_t)|\,dt<\infty$ for every $a\ge0$,
\item For every $a\ge0$,
\begin{equation}
\label{eq2.1}
Y_t=Y_{a\wedge\vartheta}+\int^{a\wedge\vartheta}_{t\wedge
\vartheta}f(r,Y_r)\,dr  -\int^{a\wedge\vartheta}_{t\wedge
\vartheta}\,dM_r,\quad t\in[0,a],
\end{equation}
\item $Y_{a\wedge\vartheta}\rightarrow\xi$ $P$-a.s.
as $a\rightarrow\infty$.
\end{enumerate}

\end{definition}

We will need the following assumptions.
%We also assume that $f$ is nonincreasing  with respect to $y$ for fixed a.e. $(\omega,t)\in \Omega\times [0,\infty)$ and $E\int_0^T|f(r,0)|\,dr<\infty$.
\begin{enumerate}
\item[(A1)] $E|\xi|<\infty$  and there exists a c\`adl\`ag process $S$
such that $S$ is a difference of two supermartingales of class (D)
and $E\int_0^{\vartheta}|f(t,S_t)|\,dt<\infty$.

\item[(A2)] For a.e. $t\in [0,\vartheta)$ the function $y\mapsto f(t,y)$ is non-increasing  $P$-a.s.

\item[(A3)] For a.e. $t\in [0,\vartheta)$ the function $y\mapsto f(t,y)$
is continuous $P$-a.s.

\item[(A4)] For every $y\in\BR$,
$\int_0^{\vartheta}|f(t,y)|\,dt<\infty$ $P$-a.s.

\item[(A5)] $L$ is a c\`adl\`ag adapted process of class (D)
such that $\limsup_{a\rightarrow\infty}L_{a\wedge\vartheta}\le\xi$,

\item[(A6)]There exists a process $U$ such that $L\le U$, $U$
is a difference of two supermartingales of class (D) and
$E\int_0^{\vartheta}f^-(t,U_t)\,dt<\infty$.
\end{enumerate}

\begin{theorem}
\label{th2.2}
\begin{enumerate}[\rm(i)]
\item Under \mbox{\rm(A2)} there exists at most one solution of
\mbox{\rm BSDE}${}^{\vartheta}(\xi,f)$.

\item If  \mbox{\rm (A1)--(A4)} are satisfied, then  there exists
a solution to \mbox{\rm BSDE}${}^{\vartheta}(\xi,f)$. Moreover, $M$ is a uniformly integrable martingale and
\begin{equation}
\label{eq2.2}
E\int^{\vartheta}_0|f(t,Y_t)|\,dt<\infty.
\end{equation}
\end{enumerate}
\end{theorem}
\begin{proof}
Part (i) is a direct consequence of \cite[Proposition
2.4]{K:SPA2}. For (ii) see  \cite[Theorem 2.9]{K:SPA2} and \cite[Remark 2.2]{K:SPA2}.
\end{proof}

%Theorem \ref{th2.2} was proved in \cite{KR:JFA} in case
%$\vartheta:=T>0$ is deterministic and the process $S$ appearing in
%condition (A1) is equal to zero. It is worth noting here that for
%some applications (A1) with $S=0$ is too strong (see \cite[Remark
%2.6]{K:SPA2}).

We now  recall  the notion of the nonlinear $f$-expectation introduced by Peng \cite{Pe1} (see also \cite{Pe2}).
For $\alpha,\beta\in\TT$ such that $\alpha\le\beta$ and  $f$
satisfying (A1)--(A4) we define the operator
\[
\EE^f_{\alpha,\beta}: L^1(\Omega,\FF_\beta,P)\rightarrow
L^1(\Omega,\FF_\alpha,P)
\]
by
\[
\EE^f_{\alpha,\beta}(\xi)= Y_\alpha,\qquad \xi\in L^1(\Omega,\FF_\beta,P),
\]
where $(Y,M)$ is the unique solution of  BSDE$^\beta(\xi,f)$. Note
that in general $\EE^f$ is a nonlinear operator. Applying It\^o's
formula shows that if $f$ is linear of the form
$f(t,y)=-c(t)y+b(t)$ for  some $\mathbb F$-adapted processes $b$
and $c\ge0$, then
\[
\EE^f_{\alpha,\beta}(\xi)=E\Big(e^{-\int_{0}^{\beta}c(r)\,dr}\xi+\int_0^\beta e^{-\int_{0}^{t}c(r)\,dr}b(t)\,dt\,\big|\,\FF_\alpha\Big).
\]

We say that a c\`adl\`ag process $X$ of class (D) is an
$\EE^f$-supermartingale (resp. $\EE^f$-submartingale) on
$[\alpha,\beta]$ if $\EE^f_{\sigma,\tau}(X_\tau)\le X_\sigma$
(resp. $\EE^f_{\sigma,\tau}(X_\tau)\ge X_\sigma$) for all
$\tau,\sigma\in \TT$ such that $\alpha\le\sigma\le\tau\le\beta$.
Of course, $X$ is called an $\EE^f$-martingale on $[\alpha,\beta]$
if it is both $\EE^f$-supermartingale and $\EE^f$-submartingale on
$[\alpha,\beta]$. For a given c\`adl\`ag process $V$ and stopping
times $\alpha,\beta$ ($\alpha\le \beta$) we denote  by
$|V|_{\alpha,\beta}$ the total variation of the process $V$ on
$[\alpha,\beta]$.

\begin{proposition}
\label{prop3.2} Assume that  $f$ satisfies \mbox{\rm(A1)--(A4)}
and let $\alpha,\beta\in\TT^{\vartheta}$ be such that $\alpha\le\beta$.
\begin{enumerate}
\item[\rm(i)] Let $\xi\in L^1(\Omega,\FF_\beta;P)$ and $G$ be a c\`adl\`ag $\BF$-adapted finite
variation process such that  $G_\alpha=0$ and
$E|G|_{\alpha,\beta}<\infty$. Then there exists a unique solution
$(X,N)$ of  BSDE$^{\alpha,\beta}(\xi,f+dG)$. Moreover,  if $G$
(resp. $-G$) is an increasing process, then $X$ is an
$\EE^f$-supermartingale (resp. $\EE^f$-submartingale) on
$[\alpha,\beta]$.

\item[\rm(ii)] If  $\xi_1,\xi_2\in L^1(\Omega,\FF_\beta,P)$ and $\xi_1\le\xi_2$, then $\EE^f_{\alpha,\beta}(\xi_1)\le\EE^f_{\alpha,\beta}(\xi_2)$.

\item[\rm(iii)] Let $f_1,f_2$ satisfy \mbox{\rm(A1)--(A4)} and $\alpha,\beta_1,\beta_2\in\TT^{\vartheta}$ be such that  $\alpha\le\beta_1\le\beta_2$. Then for any  $\xi_1\in L^1(\Omega,\FF_{\beta_1},P)$ and  $\xi_2\in L^1(\Omega,\FF_{\beta_2},P)$,
\begin{align*}
|\EE^{f_1}_{\alpha,\beta_1}(\xi_1)-\EE^{f_2}_{\alpha,\beta_2}(\xi_2)|&\le
E\Big(|\xi_1-\xi_2|+\int_\alpha^{\beta_1}|f^1(t,Y^1_t)-f^2(t,Y^1_t)|\,dt
\\
&\quad+\int_{\beta_1}^{\beta_2}|f^2(t,Y^2_t)|\,dt\,\big|\,\FF_\alpha\Big),
\end{align*}
where $Y^1_t=\EE^{f^1}_{t\wedge\beta_1,\beta_1}(\xi_1)$,
$Y^2_t=\EE^{f^2}_{t\wedge\beta_2,\beta_2}(\xi_2)$.
\end{enumerate}
\end{proposition}
\begin{proof}
Assertion (iii) follows from \cite[Theorem 2.9]{K:SPA2}  and (ii)
follows from \cite[Proposition 2.4]{K:SPA2}. The existence part in
(i) follows from \cite[Theorem 2.9]{K:SPA2}. Now assume that $X$
is as in (i) and $G$ is an increasing process. Let
$\sigma,\tau\in\TT$ be such that
$\alpha\le\sigma\le\tau\le\beta$, and let $(X^\tau,N^\tau)$ be a
solution of BSDE$^{\alpha,\tau}(X_\tau,f)$. It is clear that
$(X,N)$ is a solution of BSDE$^{\alpha,\tau}(X_\tau,f+dG)$.
Therefore, by \cite[Proposition 2.4]{K:SPA2}, $X\ge X^\tau$ on
$[\alpha,\tau]$. In particular, $X_\sigma\ge X^\tau_\sigma$.  By
the definition of the nonlinear expectation,
$\EE^f_{\sigma,\tau}(X_\tau)=X^\tau_\sigma$, so
$\EE^f_{\sigma,\tau}(X_\tau)\le X_\sigma$. A similar  reasoning
in the case where $-V$ is increasing gives the result.
\end{proof}

\begin{definition}
We say that a triple  $(Y,M,K)$ of adapted c\`adl\`ag  processes
is a solution, on the interval $[0,\vartheta]$, of the  RBSDE  with
terminal condition $\xi$,  coefficient $f$ and barrier $L$
(RBSDE${}^{\vartheta}(\xi,f,L)$ for short) if $Y$ is a  of class (D),
$M$ is a local martingale such that $M_0=0$, $K$ is an increasing
process with $K_0=0$ and the following conditions are satisfied
$P$-a.s.:
\begin{enumerate}[(a)]
\item $\int^{a\wedge\vartheta}_0|f(t,Y_t)|\,dt<\infty$ for every $a\ge0$.

\item For every $a\ge0$,
\begin{equation}
\label{eq2.3}
Y_t=Y_{a\wedge\vartheta}+\int^{a\wedge\vartheta}_{t\wedge\vartheta}
f(r,Y_r)\,dr+\int^{a\wedge\vartheta}_{t\wedge\vartheta}dK_r
-\int^{a\wedge\vartheta}_{t\wedge\vartheta}\,dM_r,\quad t\in[0,a],
\end{equation}
\item $Y_t\ge L_t$, $t\in[0,a\wedge\vartheta]$ and
$\int_0^{a\wedge\vartheta}(Y_{t-}-L_{t-})\,dK_t=0$ for every $a\ge0$,

\item $Y_{a\wedge\vartheta}\rightarrow\xi$  as $a\rightarrow\infty$.
\end{enumerate}
\end{definition}

Let  $\eta$ be  a strictly positive $\mathbb F$-progressively
measurable process such that $\eta\le 1$ and
\begin{equation}
\label{eq2.5} E\int_0^\vartheta\eta_t(S_t-L_t)^-\,dt<\infty,
\end{equation}
where $S$ is the process appearing in \mbox{\rm(A1)}.
As an example of such $\eta$ can serve any
deterministic strictly positive  bounded by $1$ process
such that $\int^{\infty}_0\eta_t\,dt<\infty$. Since $S$ and $L$
are assumed to be of class (D), $\eta$ satisfies (\ref{eq2.5}).
We let
\begin{equation}
\label{eq2.4} f_n(t,y)=f(t,y)+n\eta_t(y-L_t)^-,\quad t\ge 0,\,
y\in\BR.
\end{equation}

\begin{theorem}
\label{th2.5} Assume that \mbox{\rm (A1)--(A5)} are satisfied.
Then
\begin{enumerate}[\rm(i)]
\item There exists a unique  solution $(Y,M,K)$ to
\mbox{\rm RBSDE}${}^{\vartheta}(\xi,f,L)$.

\item Let $\xi$ be an $\FF_{\vartheta}$-measurable integrable
random variable. Then for
every $n\ge1$ there exists a unique solution  $(Y^n,M^n)$ of
\mbox{\rm BSDE}$^\vartheta(\xi,f_n)$ with $f_n$ defined by
\mbox{\rm(\ref{eq2.4})}, and $Y^n\nearrow Y$ $P$-a.s. as
$n\rightarrow\infty$.

\item If  we assume additionally  that \mbox{\rm (A6)} is satisfied,
then  $M$ is a uniformly
integrable martingale, $EK_{\vartheta}<\infty$ and
\mbox{\rm{(\ref{eq2.2})}} is satisfied.
\end{enumerate}
\end{theorem}
\begin{proof}
See \cite[Proposition A.16]{K:SPA3} and  \cite[Theorem 3.9]{K:SPA2}.
\end{proof}

For $\varepsilon>0$, we set
\begin{equation}
\label{eq5.4.af} \sigma_\varepsilon=\inf\{t\ge\alpha: Y_t\le
L_t+\varepsilon\}\wedge\vartheta.
\end{equation}

\begin{theorem}
\label{th3.3}
Assume that \mbox{\rm(A1)--(A5)} are satisfied.
Then $(Y,M,K)$ is a solution of {\rm RBSDE}$^{\vartheta}(\xi,f,L)$ if and only
if  for every $\alpha\in\TT^{\vartheta}$,
\begin{align}
\label{eq4.5f} Y_\alpha= \esssup_{\sigma\in
\TT^{\vartheta}}\EE^f_{\alpha,\sigma}
(L_\sigma\fch_{\{\sigma<\vartheta\}}+\xi\fch_{\{\sigma=\vartheta\}}).
\end{align}
\end{theorem}
\begin{proof}
Let $\sigma\in\TT^{\vartheta}$ and  $\{\delta_n\}$ be a fundamental
sequence for the local martingale $M$ on $[\alpha, \infty)$. From the minimality
condition we deduce that $(Y,M)$ is a solution of
BSDE$^{\alpha,\sigma_\varepsilon}(Y_{\sigma_\varepsilon},f)$. By
Proposition \ref{prop3.2}(i), $Y$ is an $\EE^f$-martingale on
$[\alpha,\sigma_\varepsilon]$. Hence
\[
Y_\alpha=\EE^f_{\alpha,\sigma_\varepsilon}(Y_{\sigma_\varepsilon}).
\]
On the other hand, by the definition of $\sigma_\varepsilon$ and
Proposition \ref{prop3.2}(ii), it follows  from the above
equality that
\[
Y_\alpha\le\EE^f_{\alpha,\sigma_\varepsilon}
(L_{\sigma_\varepsilon}\mathbf{1}_{\{\sigma_\varepsilon<\vartheta\}}
+\varepsilon+\xi\mathbf{1}_{\{\sigma_\varepsilon=\vartheta\}}).
\]
From this and Proposition \ref{prop3.2}(iii) we get
\begin{equation}
\label{eq5.aaggf} Y_\alpha\le \EE^f_{\alpha,\sigma_\varepsilon}
(L_{\sigma_\varepsilon}\mathbf{1}_{\{\sigma_\varepsilon<\vartheta\}}
+\xi\mathbf{1}_{\{\sigma_\varepsilon=\vartheta\}})+\varepsilon,
\end{equation}
from which one can easily deduce that
(\ref{eq4.5f}) is satisfied.
To prove the  sufficiency part, we denote by
$\bar Y_{\alpha}$ the right-hand side of (\ref{eq4.5f}). By \cite[Theorem
3.9]{K:SPA2}, there exists a unique solution $(Y,M,K)$ of
RBSDE$^{\vartheta}(\xi,f,L)$. By the necessity part in (i),
$\bar Y_\alpha=Y_\alpha$, $\alpha\in \TT^{\vartheta}$, so $(\bar Y,M,K)$ is a solution
of RBSDE$^{\vartheta}(\xi,f,L)$.
\end{proof}

\section{Stability results for solutions of RBSDEs}
\label{sec3}

In this section, we prove stability results for solutions of RBSDEs.
For given $\alpha,\beta\in\TT^{\vartheta}$ such that $\alpha\le\beta$ we set
\begin{equation}
\label{eq4.4}
\|Y\|_{1,\alpha,\beta}=\sup_{\alpha\le \tau\le\beta, \tau<\infty}E|Y_\tau|,\qquad \|Y\|_{1,\beta}=\|Y\|_{1,0,\beta}.
\end{equation}

In what follows, $L^1$ and $L^2$ are c\`adl\`ag adapted processes
of class (D).

\begin{proposition}
\label{prop3.1} Assume that $\xi^1,\xi^2$ are
$\FF_{\vartheta}$-measurable and $E|\xi^1|+E|\xi^2|<\infty$. Let
$(Y^i,M^i,K^i)$ be a solution of  \mbox{\rm
RBSDE}$^{\vartheta}(\xi^i,f^i,L^i)$, $i=1,2$,  and $f^1$  satisfy
\mbox{\rm(A2)}. Then
\[
\|Y^1-Y^2\|_{1;\alpha}\le
E|\xi^1-\xi^2|+E\int_0^{\vartheta}|f^1(t,Y^2_t)-f^2(t,Y^2_t)|\,dt
+\|L^1-L^2\|_{1;\alpha}.
\]
\end{proposition}
\begin{proof}
See \cite[Corollary 3.15, Remark 4.2]{K:SPA3}.
\end{proof}

\begin{remark}
One can get the above result by using  the  representation
(\ref{eq4.5f}) and properties of nonlinear expectation. However,
to apply this second method one has to impose   much stronger
conditions  on $f^1, f^2$. This is due to the fact that the
nonlinear expectations $\EE^{f^1}$ and $\EE^{f^2}$ are well
defined under (A1)--(A4).
\end{remark}

For a finite variation process $C$, we  denote  by $|C|_t$ its
total variation on the interval $[0,t]$.

\begin{theorem}
\label{th4.3} Let $\alpha,\beta\in \TT^{\vartheta}$ be such that
$\alpha\le \beta$ and let $\xi^1\in \FF_\alpha, \xi^2\in
\FF_\beta$ satisfy $E|\xi^1|+E|\xi^2|<\infty$.    Suppose that
$(Y^1,M^1,K^1)$ is a solution of \mbox{\rm
RBSDE}$^\alpha(\xi^1,f,L^1)$ and $(Y^2,M^2,K^2)$ is a solution of
\mbox{\rm RBSDE}$^\beta(\xi^2,f,L^2)$ with some $f$ satisfying
\mbox{\rm(A2)}. Then
\begin{enumerate}[\rm(i)]
\item $\|Y^1-Y^2\|_{1;\alpha}\le E|\xi^1-Y^2_\alpha|+\|L^1-L^2\|_{1;\alpha}$.
\item Set $\tilde Y^1_t=
Y^1_t\mathbf{1}_{[0,\alpha)}(t)+H_t\fch_{[\alpha,\beta]}(t)\fch_{\{\alpha<\infty\}}$, where  $H$ is  a c\`adl\`ag process which is a difference of two
supermartingales of class \mbox{\rm(D)} on $[0,\beta]$ with $\lim_{a\to\infty} H_{\beta\wedge a}=H_\beta$. Let  $H_t=H_0+C_t+N_t$ be  the   Doob--Meyer decomposition of $H$ ($C$ is a predictable finite variation process with $C_0=0$,
and $N$ is a martingale with $N_0=0$). Then
\begin{align}
\label{eq2.0.1}\nonumber &\|\tilde Y^1-Y^2\|_{1;\beta}
\le E|\xi^1-\xi^2|\fch_{\{\alpha=\infty\}}+
E|\xi^{1}-H_\alpha|+E|\xi^2-H_\beta|\\
&\quad+E\int_\alpha^\beta|f(t,H_t)|\,dt
+E\int_\alpha^\beta\,d|C|_t
+\|H-L^2\|_{1;\alpha,\beta}+\|L^1-L^2\|_{1;\alpha}.
\end{align}
\end{enumerate}
\end{theorem}
\begin{proof}
Observe that $(Y^2,M^2,K^2)$ is a solution of RBSDE$^\alpha(Y^2_\alpha,f,L^2)$.
Therefore (i) follows immediately from  Proposition \ref{prop3.1}.
To prove (ii), set
\[
N^\alpha_t=N_{t\vee \alpha}-N_\alpha,\qquad C^\alpha_t=C_{t\vee\alpha}-C_\alpha,
\]
\[
\tilde K^1_t=K^1_{t\wedge\alpha},\quad \tilde
M^1_t=M^1_{t\wedge\alpha}+N^\alpha_t,\quad
V^1_t=(\xi^{1}-H_\alpha)\mathbf{1}_{[\alpha,\beta]}(t)\mathbf{1}_{\{\alpha<\infty\}},
\quad t\in [0,\beta],
\]
and
\[
\tilde f(t,y):= f(t,y) \fch_{[0,\alpha]}(t),\quad  \tilde L^1_t=L^1_t\fch_{[0,\alpha)}(t)+H_t\fch_{[\alpha,\beta]}(t)\fch_{\{\alpha<\infty\}}. \quad t\in
[0,\beta], \,y\in\BR.
\]
Then $(\tilde Y^1,\tilde M^1,\tilde K^1)$ is a solution of
RBSDE$^\beta(\xi^1\fch_{\{\alpha=\infty\}}+H_\beta \mathbf{1}_{\{\alpha<\infty\}},\tilde
f+dV^1+dC^\alpha,\tilde L^1)$. Hence, by Proposition \ref{prop3.1},
\begin{align*}
\|\tilde Y^1-Y^2\|_{1;\beta}&\le
E|\xi^1\fch_{\{\alpha=\infty\}}+H_\beta\fch_{\{\alpha<\infty\}}-\xi^2|
+E\int_0^\beta\,d|V^1|_t+E\int_0^\beta\,d|C^\alpha|_t\\
&\quad+E\int_0^\beta|\tilde
f(t,\tilde Y^1_t)-f(t,\tilde Y^1_t)|\,dt+\|\tilde L^1-L^2\|_{1;\beta},
\end{align*}
which leads to (\ref{eq2.0.1}).
\end{proof}

\begin{remark}
(i) The basic difference between assertions (i) and (ii) is that unlike  (i), the right-hand side of the estimate in (ii) does not depend on the solution. This allows one to provide some results  on the rate of convergence.

(ii) At first glance the presence of a process $H$ on the right-hand  side of the inequality in (ii) is puzzling. We shall see later on that in order to get the rate of convergence in \eqref{eq1.1} it is necessary
to use (ii) with a suitable process $H$ depending on the model. In most   cases $H$ cannot be taken to be zero.
\end{remark}

\section{RBSDEs and value functions of stopping problems}
\label{sec4}

In what follows,
$\BM=(X=(X_t)_{t\ge0},(\theta_t)_{t\ge0},\BF=(\FF_t)_{t\ge0},(P_x)_{x\in
E})$ is a Borel right Markov process with state space
$E$ (augmented by a cemetery state $\partial$), shift operators
$(\theta_t)_{t\ge0}$ and life time $\zeta$, defined on some
measurable space $(\Omega,\FF)$ (see, e.g., \cite{Sharpe}).  We shall use the symbol $\BE_x$ to denote
the expectation with respect to $P_x$.
We adopt the convention that $X_\infty=\partial$.
We also adopt the convention that every function $g$ on $E$  (resp.
$E\times\BR$) is extended to $E\cup\{\partial\}$ (resp.
$(E\cup\{\partial\})\times\BR$) by setting $g(\partial)=0$ (resp.
$g(\partial,y)=0$, $y\in\BR$).

Let $D$ be an open subset of $E$ and $(P^D_t)_{t>0}$  be  the
semigroup  associated with the process $\BM$ killed when exiting  $D$, that is
\[
P^D_t\rho(x)=\mathbb E_x(\rho(X_{t})\mathbf{1}_{\{t<\tau_D\}}),
\quad t\ge0,\quad x\in D,
\]
and let $R^D$ denote the potential operator, that is
\[
R^D\rho(x)=\mathbb E_x\int^{\tau_D}_0\rho(X_t)\,dt,\quad x\in D
\]
for any positive Borel $\rho:E\rightarrow\BR$.
We assume that $(P^D_t)$ is transient, i.e. there exists a strictly positive function $\bar\rho$ such that $R^D\bar\rho$ is finite.
By the  strong Markov
property, for all  positive Borel functions
$\rho:E\rightarrow\BR$ and  $t>0$ we have
\begin{equation}
\label{eq4.6}
\mathbb E_x\int^{\tau_D}_{t\wedge\tau_D}\rho(X_s)\,ds
=P^D_tR^D\rho(x), \quad x\in D.
\end{equation}
%K:\\
%\begin{align*}
%\mathbb E_x\int^{\infty}_Te^{-\lambda t}g(X_t)\,dt&=\mathbb E_x\int^{\infty}_0e^{-\lambda(r+T)}g(X_{r+T})\,dr
%=\mathbb E_x\int^{\infty}_0e^{-\lambda(r+T)}g(X_r)\,dr\circ\theta_T\\
%&=e^{-\lambda T}\mathbb E_x\mathbb E_x(\int^{\infty}_0e^{-\lambda r}g(X_r)\,dr\circ\theta_T|\FF^0_T)\\
%&=e^{-\lambda T}\mathbb E_xE_{X_T}\int^{\infty}_0e^{-\lambda r}g(X_r)\,dr=e^{-\lambda T}\mathbb E_xR_{\lambda}g(X_T)\\
%&=e^{-\lambda T}P_TR_{\lambda}g(x).
%\end{align*}
%KK\\

For given  Borel measurable  functions  $\varphi:E\rightarrow\BR$, $\psi: D^c\to \BR$ and $g:E\times\BR\rightarrow\BR$, $h:E\to\BR$ we set
\[
\xi^{T,s}=\varphi(X_{T_s})\fch_{\{\tau_D> T_s\}}+\psi(X_{\tau_D})\fch_{\{\tau_D\le T_s\}}, \quad
\xi=\psi(X_{\tau_D})
\]
and
\[
f(\omega,t,y)=g(X_t(\omega),y),\quad t\ge0,\,y\in\BR,
\qquad  L_t=h(X_t),\quad t\ge0.
\]

We will need the following assumptions:
\begin{enumerate}
\item[(H)] For any $x\in D$ the function $y\mapsto g(x,y)$ is continuous and nonincreasing
\end{enumerate}
and, in the stationary case,
\begin{enumerate}
\item[(S1)]
$\mathbb E_x|\psi(X_{\tau_D})|+\mathbb E_x\int_0^{\tau_D}|g(X_t,0)|\,dt<\infty$,

\item[(S2)] For every $y\in\BR$, $\int^{\tau_D}_0|g(X_t,y)|\,dt<\infty$
$P_x$-a.s.

\item[(S3)]$L$ is a c\`adl\`ag process of class (D) under %the measure
$P_x$ such that
$\lim_{a\rightarrow\infty}h(X_{a\wedge\tau_D})= \psi(X_{\tau_D})$ $P_x$-a.s.

\end{enumerate}

The counterparts to (S1)--(S3) in the  evolutionary case are as follows:
\begin{enumerate}
\item[(E1)]
$\mathbb E_x|\xi^{T,s}|+\mathbb E_x\int_0^{T_s\wedge \tau_D}|g(X_t,0)|\,dt<\infty$,

\item[(E2)] For every $y\in\BR$, $\int^{T_s\wedge \tau_D}_0|g(X_t,y)|\,dt<\infty$
$P_x$-a.s.,

\item[(E3)]$L$ is a c\`adl\`ag  process on $[0,T_s\wedge\tau_D]$
of class (D) under the measure $P_x$ and  $h(X_{T_s\wedge\tau_D})\le\xi^{T,s}$
$P_x$-a.s.
\end{enumerate}

If (H) and  (S1)--(S3) are satisfied, then by Theorem \ref{th2.5} there exists a unique solution $(Y^x,M^x,K^x)$ of RBSDE$^{\tau_D}(\xi,f,L)$
under the measure $P_x$. Moreover, from  Theorem \ref{th2.5} we conclude that
if (H) and (E1)--(E3) are satisfied, then  there exists a unique solution $(Y^{T,s,x},M^{T,s,x},K^{T,s,x})$ of RBSDE$^{T_s\wedge \tau_D}(\xi^{T,s},f,L)$. We are going to show that
\begin{equation}
\label{eq4.10}
V(X)=Y^x,\qquad V_T(s+\cdot,X)=Y^{T,s,x}\quad P_x\mbox{-}a.s.,
\end{equation}
where $V_T$ is defined by (\ref{eq1.01}) and $V$ is defined by (\ref{eq1.02}). This together with the stability results for  RBSDEs proved in Theorem \ref{th4.3}
yields the main result of the paper.
First, however, we shall prove a weaker result.

\begin{proposition}
\label{prop4.5}
Let $x\in D$ and $s\ge 0$.  Assume that  \mbox{\rm(H)} and \mbox{\rm(S1)--(S3)} are satisfied,  and for every $T>0$,  \mbox{\rm(E3)}  is satisfied  and
$\BE_x|\varphi(X_T)|<\infty$.
\begin{enumerate}[\rm(i)]
\item If
\begin{equation}
\label{eq4.2}
\lim_{a\rightarrow\infty}\BE_x\big(|\varphi(X_a)|\mathbf1_{\{a<\tau_D\}}\big)=0,
\end{equation}
then $\lim_{T_s\rightarrow\infty}V_T(s,x)=V(x)$.

\item Let $\gamma$ be a Borel function such that $\gamma(X)$
is a difference of two supermartingales of class \mbox{\rm(D)} on $[0,\tau_D]$ and  $\gamma(X_{\tau_D})=\psi(X_{\tau_D})$. Then  for any $0\le s\le T$,
\begin{align}
\label{eq4.7.3}
\nonumber |V_T(s,x)-V(x)|&\le \BE_x\big(|\gamma-\varphi|(X_{T_s})\fch_{\{\tau_D> T_s\}}\big)
+ \BE_x\int_{T_s\wedge\tau_D}^{\tau_D}|g(X_t,\gamma(X_t))|\,dt\\
&\quad+\BE_x\int_{T_s\wedge\tau_D}^{\tau_D}\,d|C^x|_t+
\sup_{T_{s}\wedge \tau_D\le\tau\le \tau_D}\BE_x|\gamma(X_{\tau})-h(X_{\tau})|,
\end{align}
where $C^x$ is the predictable finite variation part of the Doob-Meyer decomposition of $\gamma(X)$ with $C^x_0=0$.
%\begin{align}
%\label{eq4.7}
%|V_T(x)-V(x)|&\le P_T|\varphi|(x) + P_T(R|f(\cdot,0)|)(x)
%\nonumber\\
%&\quad+\sup_{\tau\ge T}\mathbb E_x|g_{\infty}(X_{\tau})| %+\sup_{\tau\in\TT_T}\mathbb E_x|g(\tau,X_{\tau})-g_{\infty}(X_{\tau})|.
%\end{align}
\end{enumerate}
\end{proposition}
\begin{proof}
By Theorem \ref{th3.3}, $V(x)=\mathbb E_xY^x_0$ and $V_T(s,x)=\mathbb E_x Y^{T,s,x}_0$.
Therefore, by Theorem \ref{th4.3}(i) applied to $Y^1=Y^{T,s,x}$, $Y^2=Y^x$ and $\alpha=T_s\wedge\tau_D$, $\beta=\tau_D$, we have
\begin{align*}
|V_T(s,x)-V(x)|&\le\BE_x|\xi^{T,s}-Y^x_{T_s\wedge\tau_D}|\\
&=\BE_x|\varphi(X_{T_s})\fch_{\{\tau_D>T_s\}}+\psi(X_{\tau_D})\fch_{\{\tau_D\le T_s\}}
- Y^x_{T_s\wedge\tau_D}|.%\\
%&\le\BE_x|\varphi(X_{T_s})- Y^x_{T_s}|\fch_{\{\tau_D=\infty\}}\\
%&+\BE_x|\varphi(X_{T_s})\fch_{\{\tau_D>T_s\}}+\psi(X_{\tau_D})\fch_{\{\tau_D\le T_s\}}
%- Y^x_{T_s\wedge\tau_D}|\fch_{\{\tau_D<\infty\}}
\end{align*}
Since $Y^x$ is of class (D) under the measure $P_x$, and $Y^x_{T_s\wedge\tau_D}\rightarrow\xi=\psi(X_{\tau_D})=\psi(X_{\tau_D})\mathbf1_{\{\tau_D<\infty\}}$ $P_x$-a.s. as $T_s\rightarrow\infty$,
then $\mathbb E_x|\psi(X_{\tau_D})\fch_{\{\tau_D\le T_s\}}
- Y^x_{T_s\wedge\tau_D}|\to 0$ as $T_s\to \infty$. From this and \eqref{eq4.2} we obtain at once that the  right-hand side of the above inequality converges to zero as $T_s\rightarrow\infty$.
This proves (i). Part (ii) follows from Theorem \ref{th4.3}(ii) with $H=\gamma(X)$ and
$Y^1,Y^2$, $\alpha,\beta$ as above.
\end{proof}

Recall  that $D$ is called Dirichlet regular if
$P_x(\tau_D>0)=0$ for all $x\in \partial D$.

\begin{lemma}
\label{lem4.2} Let $\gamma$ be a positive Borel function on $E$ and  $v(x)=\mathbb E_x\gamma(X_{\tau_D})$, $x\in E$.
Then for every  $\alpha\in\TT^{\tau_D}$,
\[
\fch_{\{\alpha<\tau_D\}}v(X_\alpha)=\fch_{\{\alpha<\tau_D\}}
\BE_x( \gamma(X_{\tau_D})|{\FF_\alpha}),\quad P_x\mbox{-a.s.},\quad
x\in E.
\]
Moreover, if $D$ is Dirichlet regular, then for every
$\alpha\in\TT^{\tau_D}$ we have
\[
v(X_\alpha)=\mathbb E_x(\gamma(X_{\tau_D})|{\FF_\alpha}),
\quad P_x\mbox{-a.s.},\quad x\in E.
\]
\end{lemma}
\begin{proof}
Let $A=\{\alpha<\tau_D\}$, $B=\{\alpha=\tau_D\}$. All the
following equations hold $P_x$-a.s. for $x\in E$. By the strong
Markov property,
\begin{align}
\label{eq4.3}
v(X_\alpha)=E_{X_\alpha}\gamma(X_{\tau_D})
&=\BE_x(\gamma(X_{\tau_D}\circ\theta_\alpha)|{\FF_\alpha})\nonumber\\
&=\BE_x(\fch_A\gamma(X_{\tau_D}\circ\theta_\alpha)|{\FF_\alpha})
+\BE_x(\fch_B\gamma(X_{\tau_D}\circ\theta_\alpha)|{\FF_\alpha}).
\end{align}
On the set $A$ we have $\tau_D\circ
\theta_\alpha=\tau_D-\theta_\alpha$, so
$\fch_A\gamma(X_{\tau_D}\circ\theta_\alpha)=\fch_A\gamma(X_{\tau_D})$.
Therefore (\ref{eq4.3}) implies the first assertion. To prove the
second one, it suffices now to observe  that by the Dirichlet
regularity of $D$,
$\fch_B\gamma(X_{\tau_D}\circ\theta_\alpha)=\fch_B\gamma(X_{\tau_D})$.
\end{proof}

\begin{corollary}
\label{wn6.rep}
Let assumption  \mbox{\rm(H)} hold  and  \mbox{\rm(S1)--(S2)}
be satisfied  for every $x\in D$. Then for every $x\in D$
there exists a unique solution $(Y^x,M^x)$ of {\rm BSDE}$^{\tau_D} (\xi,f)$. Furthermore,
the function $u(x):=\mathbb E_xY^x_0$ is  Borel measurable, and
\begin{equation}
\label{eq6.strt12}
Y^x_t=u(X_t),\quad t<\tau_D,\quad x\in D.
\end{equation}
If $D$ is Dirichlet regular, then the above equation holds for all
$t\le\tau_D$ and $x\in D$.
\end{corollary}
\begin{proof}
Let $v$ be as  in Lemma \ref{lem4.2}. By  Lemma \ref{lem4.2}, and
a simple calculation, we have $Y^x=\bar Y^x+v(X)$, where $(\bar
Y^x,\bar M^x)$ is a solution of BSDE$^{\tau_D}(0,f_v)$ with
\[
f_v(t,y)=f(t,y+v(X_t)).
\]
By \cite[Theorem 4.7]{KR:JFA}, there exists a Borel function $\bar
u$  such that $\bar Y^x= \bar u(X)$, $x\in E$. Thus, we have
(\ref{eq6.strt12}) with $u=\bar u+v$.
\end{proof}

In the next theorem, we give a precise meaning of (\ref{eq4.10}) and give
conditions ensuring that it is satisfied.

\begin{theorem}
\label{th4.4}
Let assumption  \mbox{\rm (H)} hold.
\begin{enumerate}[\rm(i)]
\item If \mbox{\rm(S1)--(S3)} are satisfied for all $x\in D$, then
 $V(X_t)=Y^x_t$, $t\in [0,\tau_D)$, $P_x$-a.s. for $x\in D$.

\item If \mbox{\rm(E1)--(E3)} are satisfied for all  $x\in D$ and $s\in [0,T)$, then $V_T(s+t,X_{t})=Y^{T,s,x}_t$, $t\in [0,T_s\wedge\tau_D)$,
$P_x$-a.s. for $(s,x)\in [0,T)\times D$.
\end{enumerate}
Moreover, if $D$ is Dirichlet regular, then the assertions of \mbox{\rm (i)}
and \mbox{\rm(ii)} hold on the random intervals $[0,\tau_D]$ and $[0,T_s\wedge\tau_D]$, respectively.
\end{theorem}
\begin{proof}
(i)  Let $f_n(t,y)=f(t,y)+n\rho(X_t)(y-L_t)^-$, $n\ge1$ with $\rho$ being a strictly positive bounded Borel function
such that $R^D\rho$ is bounded (it exists since we assumed that $(P^D_t)$ is transient).
By Theorem \ref{th2.5}, for every $x\in D$,
\[
Y^{n,x}_t\nearrow Y^x_t,\quad t\in [0,\tau_D],\quad P_x\mbox{-a.s.,}
\]
where $(Y^{n,x},M^{n,x},K^{n,x})$ is the unique solution of BSDE$^{\tau_D}(\xi,f_n)$ under the measure $P_x$. By Corollary \ref{wn6.rep}, there exists a Borel function $u_n$
such that $u_n(X_t)=Y^{n,x}_t$, $t\in [0,\tau_D)$ (for $t\in [0,\tau_{D}]$ in case $D$ is Dirichlet regular) $P_x$-a.s.  Clearly $u_n(x)=E_xY^{n,x}_0\nearrow E_xY^x_0=:u(x),\, x\in D$. Thus $u(X_t)=Y^{x}_t$, $t\in [0,\tau_D)$ (for $t\in [0,\tau_{D}]$ in case $D$ is Dirichlet regular) $P_x$-a.s.  By Theorem \ref{th3.3}, $u=V$ on $D$.

(ii) Let  $\upsilon$ be the uniform motion to the right, that is $\upsilon(0)=s$ and $\upsilon(s)=s+t$ , $t\ge0$, under the measure $P_{s,x}$. Set $\mathscr
X_t=(\upsilon(t),X_{\upsilon(t)})$. Then $\{(\mathscr X,P_{s,x}),
(s,x)\in\BR_+\times D\}$, where $P_{s,x}(\mathscr X_t\in
A)=P_{x}((t+s,X_t)\in A)$ for  any Borel subset of $\BR_+\times
D$, is a Markov process with state space  $\BR_+\times D$ (see,
e.g., \cite[Section 8.5.5]{Wentzell}). Set $\hat D=[0,T)\times D$,
$\tau_{\hat D}=\inf\{t>0: \mathscr X_t\notin \hat D\}$ and
$\hat\xi^T=\hat\psi(\mathscr X_{\tau_{\hat D}})$, where
\[
\hat\psi(t,x)=\varphi(x)\fch_{\{t=T,x\in D\}}+\psi(x)\fch_{\{t<T,x\notin D\}}.
\]
We also set $\hat f(t,y)=g(\Pi(\mathscr X_t),y)$, $\hat L_t=h(\Pi(\mathscr X_t))$, where $\Pi$
denotes the  canonical projection on $E$. By Theorem \ref{th2.5}, for every $(s,x)\in \hat D$ there exists a unique solution
$(\hat Y^{s,x},\hat M^{s,x},\hat K^{s,x})$ of RBSDE$^{\tau_{\hat D}}(\hat \xi^T,\hat f,\hat L)$ under $P_{s,x}$. Moreover,
\[
\hat Y^{n,s,x}_t\nearrow \hat Y^{s,x}_t,\quad t\in [0,\tau_{\hat D}],
\]
where $(\hat Y^{n,s,x},\hat M^{n,s,x})$ is the unique solution to BSDE$^{\tau_{\hat D}}(\hat\xi^T,\hat f_n)$
with $\hat f_n(t,y)=\hat f(t,y)+n(y-\hat L_t)^-$.  By Corollary \ref{wn6.rep}, there exists a Borel function $u_n$ on $\hat D$
such that $u_n(\mathscr X_t)=\hat Y^{n,s,x}_t$, $t\in [0,\tau_{\hat D})$, $P_{s,x}$-a.s. (for $t\in [0,\tau_{\hat D}]$ in case $D$ is Dirichlet regular). Clearly $u_n(s,x)=\BE_{s,x}\hat Y^{n,s,x}_0$.
%$\nearrow E_{s,x}\hat Y^{s,x}_0=:w(s,x),\, x\in D$. Thus, $w(\mathscr X_t)=\hat Y^{s,x}_t,\, t\in [0,\tau_{\hat D})$ a.s.
However, by the relation between $P_x$ and $P_{s,x}$, we have
\begin{align*}
\mathbb E_xu_n(s,X_0)=u_n(s,x)=\mathbb E_{s,x}\hat Y^{n,s,x}_0& =\mathbb E_x\big(\varphi(X_{T-s})\fch_{\{T_s\wedge\tau_D\}}
+\psi(X_{\tau_D})\fch_{\{\tau_D< T_s\}}\big)\\
&\quad+\BE_x\int_0^{T_s\wedge\tau_D}g(X_t,u_n(s+t,X_t))\,dt\\
&\quad +n\BE_x\int_0^{T_s\wedge\tau_D}(u_n(s+t,X_t)-h(X_t))^-\,dt.
\end{align*}
Set $f_n(t,y)=g(X_t,y)+n(y-h(X_t))^-$.  Using the strong  Markov property of $X$ we deduce that
$u_n(s+t,X_t)=Y^{n,T,s,x}_t,\, t\in [0,T_s\wedge\tau_D)$ (for $t\in [0,T_s\wedge\tau_D]$ in case of Dirichlet regular $D$), where  $(Y^{n,T,s,x},M^{n,T,s,x})$ is the unique solution of
BSDE$^{T_s\wedge\tau_D}(\xi^{T,s},f_n)$. By Theorem \ref{th2.5}, $Y^{n,T,s,x}\nearrow Y^{T,s,x}$ $P_x$-a.s.
In particular,
$u_n(s,x)=\BE_{x}Y^{n,T,s,x}_0\nearrow\BE_xY^{T,s,x}_0=: u(s,x)$.
Therefore $u(s+t,X_t)=Y^{T,s,x}_t$, $t\in [0,T_s\wedge\tau_D)$ (for $t\in [0,T_s\wedge\tau_D]$ in case of Dirichlet regular $D$), $P_x$-a.s. On the other hand, by
Theorem \ref{th3.3}, $u=V_T$ on $\hat D$.
\end{proof}

\begin{remark}
Let  (H) hold  and  (S1)--(S3) be satisfied for every $x\in E$. Let   $Y^x$ be the first component of the solution of RBSDE$^{\tau_D}(\xi,f,L)$. By Theorem \ref{th4.4}(i), $Y^x=V(X)$, so from  \cite[Remark 3.6]{K:SPA2} it follows that for any stoppong time $\alpha\le\tau_D$,
\begin{align*}
V(X_{\alpha})&=\esssup_{\sigma\ge\alpha}\mathbb E_x\Big(\int^{\sigma\wedge\tau_D}_{\alpha}g(X_r,V(X_r))\,dr \\
&\qquad\qquad+h(X_{\sigma})\fch_{\{\sigma<\tau_D\}}
+\psi(X_{\tau_D})\fch_{\{\sigma\wedge\tau_D=\tau_D\}}\,\big|\,\FF_{\alpha}\Big).
\end{align*}
Letting $\alpha=0$ we get (\ref{eq1.3}). Similarly, if (H), (S1)--(S3) are satisfied, then from Theorem \ref{th4.4}(ii) we get (\ref{eq1.2}).  Note also that from \cite[Proposition 4.3, Proposition 5.1]{KR:EJP}
it follows that $V:D\rightarrow\BR$ satisfying (\ref{eq1.3}) is unique and $V_T:[0,T]\times D\rightarrow\BR$ satisfying (\ref{eq1.2}) is unique.
%\\
%K:\\
%If $V$ satisfies (\ref{eq1.3}), then clearly $V(X)$ is a solution of the equation
%\[
%Y_t=\esssup_{\sigma\ge t}E\Big(\int^{\sigma}_tdV_t++h(X_{\sigma})\fch_{\{\sigma<\tau_D\}}
%+\psi(X_{\tau_D})\fch_{\{\sigma\wedge\tau_D=\tau_D\}}\,\big|\,\FF_t\Big)
%\]
%with $V_t=\int^{t\wedge\tau_D}_0g(X_t,V(X_t))\,dt$. Therefore, by \cite[Proposition %5.1]{KR:EJP}, $V(X)$ is a solution of the linear %RBSDE${}^{\tau_D}(\psi(X_{\tau_D)},dV,h(X))$. From the definition of $V_t$, it is a %solution of semilinear BSDE, whose solution is unique by \cite[Proposition 4.3]{KR:JEP}.\\
%KK
\end{remark}

From Theorem  \ref{th4.3} and Theorem \ref{th4.4} we get  the
dynamic version of (\ref{eq4.7.3}). It is one of the main results
of our paper.
\begin{theorem}
\label{th4.5}
%Suppose that $D$ is Dirichlet regular.
 Assume that   \mbox{\rm(H)} hold and for every  $x\in D$   assumptions \mbox{\rm(S1)--(S3)} are satisfied  and $\BE_x|\varphi(X_T)|<\infty$, $T>0$.
Moreover, assume that  \mbox{\rm(E3)} is satisfied for every $(s,x)\in [0,T)\times D$.
\begin{enumerate}
\item[\rm(i)]
If \eqref{eq4.2} is satisfied, then
\[
\lim_{T_s\rightarrow\infty}\sup_{\alpha\in \TT^{ T_s\wedge \tau_D} }\mathbb E_x\big[\mathbf1_{\{\alpha<T_s\wedge\tau_D\}}|V_T(s+\alpha,X_\alpha)-V(X_\alpha)|\big]=0.
\]
If $D$ is Dirichlet regular, then the above convergence holds true without the factor  $\fch_{\{\alpha<T_s\wedge\tau_D\}}$.

\item[\rm(ii)]Let $\gamma$, $C^x$ be as in Proposition \ref{prop4.5}. For every $(s,x)\in [0,T)\times D$,
\begin{align}
\label{eq4.7}
&\sup_{\alpha\in \TT^{ T_s\wedge \tau_D} }\mathbb E_x\big[\mathbf1_{\{\alpha<T_s\wedge\tau_D\}}|V_T(s+\alpha,X_\alpha)-V(X_\alpha)|\big]
\nonumber\\
&\qquad\le \BE_x| \gamma-\varphi|(X_{T_s})\fch_{\{\tau_D>T_s\}}
+ \BE_x\int_{T_s\wedge\tau_D}^{\tau_D}|g(X_t,\gamma(X_t))|\,dt\nonumber\\
&\qquad\quad+\BE_x\int_{T_s\wedge\tau_D}^{\tau_D}\,d|C^x|_t+
\sup_{\tau\in \TT_{T_{s}\wedge \tau_D}^{\tau_D}}\BE_x|\gamma(X_{\tau})-h(X_{\tau})|.
\end{align}
If $D$ is Dirichlet regular, then the above estimate holds true without the factor  $\fch_{\{\alpha<T_s\wedge\tau_D\}}$ on the left-hand side of \mbox{\rm(\ref{eq4.7})}.
%predictable finite variation part of the Doob-Meyer decomposition of $\gamma(X)$.
\end{enumerate}
\end{theorem}
\begin{proof}
Follows from  Theorem \ref{th4.4} and  Theorem \ref{th4.3}  applied to
$Y^1=Y^{T,s,x}, Y^2=Y^x$, $H=\gamma(X)$ and $\alpha=T_s\wedge\tau_D$, $\beta=\tau_D$ (see the proof of Proposition \ref{prop4.5}).
\end{proof}

\begin{remark}
\label{rem4.6} It is well known (see, e.g., \cite[Remark
2.1]{KRS:BSM}) that for every $q\in(0,1)$,
\begin{align*}
& \mathbb E_x\sup_{t\le T_s\wedge\tau_D} |V_T(s+t,X_t)-V(X_t)|^q\nonumber \\
&\qquad\le\frac{1}{1-q}\Big(\sup_{\alpha\in \TT^{ T_s\wedge \tau_D}}\mathbb E_x|V_T(s+\alpha,X_\alpha)-V(X_\alpha)|\Big)^q.
\end{align*}
This together with (\ref{eq4.7}) yields the rate of convergence of
the value function in the supremum norm.
\end{remark}

\section{Rate of convergence}
\label{sec5}

Let %$V^*_{T_s}$ denote the quantity  appearing on left-hand side
%of (\ref{eq4.7}) in case of Dirichlet regular $D$, i.e.
\[
V^*_{T_s}(x)=\sup_{\alpha\in \TT^{ T_s\wedge
\tau_D}}\BE_x|V_T(s+\alpha,X_\alpha)-V(X_\alpha)|.
\]
The aim of this section is to provide the rate of convergence of
$V^*_{T_s}$ as $T_s\rightarrow\infty$. To this end, we shall
estimate the right-hand side of (\ref{eq4.7}). We begin with some
general remarks. Then we discuss in more detail some specific
situations.

Throughout this section, we  assume that (S1)--(S3) and  (E1)--(E3)
are satisfied for all $x\in E$ and $s\in [0,T],\, T\ge 0$.
Moreover, we  assume that $D$ is Dirichlet regular and $\gamma$
is of the form
\[
\gamma(x)=\mathbb E_x\psi(X_{\tau_D}),\quad x\in D.
\]
Under the measure $P_x$ the process $\gamma(X)$ is a martingale on
$[0,\tau_D]$ (see the argument in the proof of Lemma
\ref{lem4.2}), so $C^x=0$. Therefore, by
Theorem \ref{th4.5},
\begin{align*}
V^*_{T_s}(x)&\le \BE_x| \gamma-\varphi|(X_{T_s})
|\fch_{\{\tau_D\ge T_s\}}
+ \BE_x\int_{T_s\wedge\tau_D}^{\tau_D}|g(X_t,\gamma(X_t))|\,dt\nonumber\\
&\quad+
\sup_{\tau\in \TT_{T_{s}\wedge \tau_D}^{\tau_D}}\BE_x|\gamma(X_{\tau})-h(X_{\tau})|.
\end{align*}
This together with (\ref{eq4.6}) gives
\begin{equation}
\label{eq55.1} V^*_{T_s}(x)\le P^D_{T_s}| \gamma-\varphi|(x)
+P^D_{T_s}R^D(|g(\cdot,\gamma)|)(x)
+\sup_{\tau\in
\TT_{T_{s}\wedge \tau_D}^{\tau_D}}
\BE_x|\gamma(X_{\tau})-h(X_{\tau})|.
\end{equation}
Define $w:D\rightarrow\BR$ by
\begin{equation}
\label{eq.dfofw}
w(x):=|\gamma-\varphi|(x)+R^D(|g(\cdot,\gamma)|)(x)
\end{equation}
and $\hat h:E\times \TT\to\BR$ by
\begin{equation}
\label{eq.dfofhh}
\hat h(x,\tau):=\mathbb E_x|\gamma(X_\tau)-h(X_\tau)|=\BE_x|\mathbb E_x(h(X_{\tau_D})|\FF_\tau)-h(X_\tau)|.
\end{equation}
With this notation (\ref{eq55.1}) can be rewritten in the form
\begin{align}
\label{eq55.22} V^*_{T_s}(x)\le
P^D_{T_s}w(x)+\sup_{\tau\in\TT^{\tau_D}_{\tau_D\wedge T_s}}\hat
h(x,\tau).
\end{align}
The rate of convergence of the first  term on the right-hand side
of \eqref{eq55.22} depends on the rate of decay of the semigroup
$(P^D_t)$ as $t\to \infty$. There are  various results in the
literature concerning this issue and we shall indicate a few of
them. However,  it is by no means  clear how to control the rate
of decay of the second term on the right-hand side of
\eqref{eq55.22}.

\subsection{General Markov processes}
\label{sec5.1}

%Define $w,\hat h$ as at the beginning of this section and
Set
\[
\hat h_\infty(x):=\sup_{\tau\in\TT^{\tau_D}} \hat h(x,\tau), \quad
x\in D.
\]

\begin{theorem}
\label{th5.1}
For all $0\le s\le T$ and $x\in D$,
\begin{equation}
\label{eq5.15}
V^*_{T_s}(x)\le P^D_{T_s}w(x)+P^D_{T_s}\hat h_\infty(x).
\end{equation}
\end{theorem}
\begin{proof}
Let $Y^x$ be the first component of the solution of RBSDE$^{\tau_D}(0,0,|\gamma-h|(X))$
under the measure $P^D_x$.
By Theorem \ref{th4.4} and Theorem \ref{th3.3},
\begin{equation}
\label{eq5.5}
\esssup_{\tau\in \TT^{\tau_D}_{T_s\wedge\tau_D}}
\BE^D_x\big(|\gamma(X_\tau)-h(X_\tau)|\big|\FF_{T_s\wedge\tau_D}\big)=Y^x_{T_s\wedge\tau_D}
=u(X_{T_s\wedge\tau_D})\quad P^D_x\mbox{-a.s.},
\end{equation}
where $\BE^D_x$ denotes the expectation with respect to $P^D_x$ and
\begin{equation}
\label{eq5.6}
u(x)=\sup_{\tau\in \TT^{\tau_D}}\mathbb E^D_x|\gamma(X_\tau)-h(X_\tau)|,\quad x\in D.
\end{equation}
Taking the expectation of both sides of \eqref{eq5.5} we get
\begin{align*}
\sup_{\tau\in \TT^{\tau_D}_{T_s\wedge\tau_D}}\mathbb E^D_x|\gamma(X_\tau)-h(X_\tau)|&=\mathbb E^D_x\esssup_{\tau\in \TT^{\tau_D}_{T_s\wedge\tau_D}}\mathbb E^D_x\big(|\gamma(X_\tau)-h(X_\tau)|\big|\FF_{T_s\wedge\tau_D}\big)\\
&=\mathbb E^D_xu(X_{T_s\wedge\tau_D})=P^D_{T_s}u(x).
\end{align*}
On the other hand, by \eqref{eq5.6}, $u(x)=\sup_{\tau\in\TT^{\tau_D}}\hat h(x,\tau)=\hat h_\infty(x)$. This together with (\ref{eq55.22})  proves (\ref{eq5.15}).
\end{proof}

By Theorem \ref{th5.1} the problem of the rate of convergence of
$V^*_{T_s}$ as $T_s\to\infty$ is reduced to the problem of the rate
of decay of the semigroup $(P^D_t)$ as $t\to \infty$. Some simple
consequences of this fact are given below.

\begin{corollary}
\label{cor5.1} Assume that $w$ and $|h-\gamma|$ are bounded. Then
for all $0\le s\le T$ and $x\in D$,
\[
V^*_{T_s}(x)\le (\|w\|_\infty+\|h-\gamma\|_\infty)P^D_{T_s}1(x)
=(\|w\|_\infty+\|h-\gamma\|_\infty)P_x(\tau_D\ge T_s).
\]
\end{corollary}
\begin{proof}
Follows immediately form Theorem \ref{th5.1} since $\|\hat
h_{\infty}\|_{\infty}\le\|\gamma-h\|_{\infty}$.
\end{proof}

In applications very often it is known that $h(X)$ is a
supermartingale or submartingale. In such a case the following
lemma is useful.
%and if this is not the case, we can always find a
%supermartingale which dominates $h(X)$ (by using, for example,
%Snell envelope).

\begin{lemma}
\label{lem5.1} Assume that $x\in D$ and  $h(X)$ is a submartingale (or
supermartingale) under the measure $P_x$. Then
\[
\hat h_\infty(x)= |\gamma-h|(x).
\]
\end{lemma}
\begin{proof}
Clearly $\hat h_\infty(x)\ge |\gamma-h|(x),\, x\in D$.  By the
assumption that $h(X)$ is a submartingale, (S3) and the definition
of $\gamma$, for $\tau\le\tau_D$ we have
\begin{align*}
\gamma(X_\tau)=\BE_x(\psi(X_{\tau_D})|\FF_\tau)=\BE_x
(h(X_{\tau_D})|\FF_\tau)\ge h(X_\tau)\quad P_x\mbox{-a.s}.
\end{align*}
From this and the fact that  $\gamma(X)$ is a martingale, we  infer
that for every $\tau\in\TT^{\tau_D}$,
\begin{equation*}
\BE_x|\gamma(X_\tau)-h(X_\tau)|=\BE_x\gamma(X_\tau)-\BE_xh(X_\tau)\le
\BE_x\gamma(X_0)- \mathbb E_x h(X_0)=|\gamma-h|(x).
\end{equation*}
An analogous reasoning applies to the case where  $h(X)$ is  a
supermartingale.
\end{proof}

\begin{corollary}
\label{cor5.2} Assume that $h(X)$ is a submartingale (or
supermartingale) under  the measure $P_x$. Then for all $0\le s\le
T$ and $x\in D$,
\[
V^*_{T_s}(x)\le P^D_{T_s}(w+|\gamma-h|)(x).
\]
\end{corollary}
\begin{proof}
Follows immediately from Theorem \ref{th5.1} and Lemma
\ref{lem5.1}.
\end{proof}

Let $m$ be a positive Borel measure on $E$ with full support.  In
the rest of this subsection, we assume that $\BM^D$ has the
transition density $p_D(t,x,y)$ with respect to $m$, that is
\[
P^D_tf(x)=\int_D f(y) p_D(t,x,y)\,m(dy),\quad x\in D,\, f\in\BB_b(D).
\]
For $q\in [1,\infty]$, we let
\begin{equation}
\label{eq5.ffrvi1}
r_q(t,x):=\|p_D(t,x,\cdot)\|_{L^q(D;m)},\quad x\in D,\, t>0.
\end{equation}
Note that $r_1(t,x)=P^D_t1(x)$. For $q\in [1,\infty]$ we let
$q^*=\frac{q}{q-1}$ if $q\in (1,\infty)$, $q^*=\infty$ if $q=1$,
and $q^*=1$ if $q=\infty$.

\begin{proposition}
\label{prop5.5} For all $0\le s\le T$, $q\in [1,\infty]$ and $x\in
D$,
\[
V^*_{T_s}(x)\le r_{q^*}(T_s,x)(\|w\|_{L^q(D;m)}+\|\hat h_\infty\|_{L^q(D;m)}).
\]
\end{proposition}
\begin{proof}
Follows  easily from Theorem \ref{th5.1} by applying H\"older's
inequality.
\end{proof}

\begin{corollary}
\label{cor5.3}
Assume that $h(X)$ is a submartingale (or supermartingale) under the measure $P_x$.
Then for all $0\le s\le T$, $q\in [1,\infty]$ and $x\in D$,
\[
V^*_{T_s}(x)\le  r_{q^*}(T_s,x)(\|w\|_{L^q(D;m)}+\|h-\gamma\|_{L^q(D;m)}) .
\]
\end{corollary}
\begin{proof}
Follows from Proposition \ref{prop5.5} and Lemma \ref{lem5.1}.
\end{proof}

\subsection{L\'evy-type operators}
\label{sec5.2}

Let $L$ be an integro-differential operator defined for $u\in
C^2(\BR^d)\cap C_b(\BR^d)$ by
\begin{align}
\label{eq.opa} \nonumber Lu(x)&=\mbox{Tr}(Q(x)\nabla^2u(x))
+b(x)\cdot\nabla u(x)-c(x)u(x)\\
&\quad+\int_{\BR^d} \left(u(x+y)-u(x)-y\cdot\nabla u(x)
 \fch_{\{|y|\le 1\}}\right)N(x,dy).
\end{align}
We assume that its coefficients $q_{ij}$, $b_i$, $c$,
$i,j=1,\dots,d$ are bounded Borel measurable functions on $\BR^d$,
$c$ is nonnegative and the matrix $ Q(x)=[q_{ij}(x)]_{i,j=1}^{d}$
is symmetric and positive definite for every $x\in \BR^d$. As
for $N(x,dy)$, we assume that it is a {\em L\'evy kernel}, that is
$N(x,dy)$ is a $\sigma$-finite positive  Borel measure on
$\BR^d\setminus\{0\}$ for each $x\in \BR^d$, and
\[
\sup_{x\in\BR^d}\int_{\BR^d\setminus\{0\}}(1\wedge|y|^2)\,N(x,dy)<\infty,\quad x\in\BR^d.
\]

Let  $\mu$  be  a  probability measure on $\BR^d$. Recall that a probability measure $P_{\mu}$ on the Skorokhod space ${\cal D}$ of c\`adl\`ag functions on $[0,\infty)$ is called a solution  of the martingale problem, in the sense of Stroock and Varadhan, associated with  the operator $L$ and initial measure $\mu$, if for every $f\in C^2_b(\BR^d)$,
\begin{equation}
\label{Mtf}
M^f_t:= f(X_t)-f(X_0)-\int_0^tLf(X_r)\,dr,\quad t\ge 0,
\end{equation}
is a  martingale under the measure $P_{\mu}$, and  $P_\mu(X_0\in B)=\mu(B)$, $B\in\mathcal B(\BR^d)$.
In what follows we assume that there exists a strong Markov process $\BM=\{(X,P_x),x\in \BR^d\}$
with the property that  for every probability measure $\mu$ on $\BR^d$ the measure $P_\mu(\cdot):=\int_{\BR^d}P_x(\cdot)\,\mu(dx)$ is a solution of the martingale problem associated with  the operator $L$ and initial measure $\mu$.
Any Markov process $\BM$ enjoying the above properties is called a strong Markov solution of the martingale problem associated with $L$.

By \cite[Theorem 4.1]{kuhn} (see also \cite[Section 6.1]{KK:B}) there exists  a strong
Markov solution of the martingale problem associated with the operator $L$ provided that for every probability measure $\mu$ on $\BR^d$
there exists a solution of the martingale problem for $L$ and initial measure $\mu$.

\begin{remark}
\label{rem5.7}
Consider the following  hypotheses:
\begin{itemize}
\item[(M1)]
The matrix
$Q(x)$  is {\em uniformly strictly positive
definite} on compact sets, i.e. for any compact set $K\subset \BR^d$
there exists $\lambda_K>0$ such that
\[
\lambda_K|\xi|^2\le \sum_{i,j=1}^dq_{ij}(x)\xi_i\xi_j,\quad x\in K,\, \xi=(\xi_1,\ldots,\xi_d)\in\BR^d.
\]
\item[(M2)] The mapping $\BR^d\ni x\mapsto Q(x)\in\BR^{d\times d}$ is continuous
and for every Borel set $B\subset B(0,1)=\{y\in\BR^d:|y|<1\}$ the mapping
\[
\BR^d\ni x\mapsto N_B(x):=\int_{B}\min\{|y|^2, 1\} \,N(x,dy)
\]
is continuous.
\end{itemize}
Let the assumptions on the coefficients of $L$ made after (\ref{eq.opa}) be satisfied.
Then there exists a strong Markov solution of the martingale problem associated with $L$ if
(M1) is satisfied (see \cite{AP,LM}) or (M2) is satisfied  and $ Q(x)$ is invertible for every $x\in\BR^d$ (see \cite[Theorem III.2.34, p. 159]{JS} and also \cite{St})
or (M2) is satisfied and the mapping $x\mapsto b(x)$ is continuous (see \cite{Hoh} and also \cite[Theorem 3.24]{BSW}).
\end{remark}

Recall here that a martingale problem is said to be {\em well posed} if for every probability measure $\mu$ on $\BR^d$
there exists a unique solution $P_\mu$ of the martingale problem associated with $L$ and initial measure $\mu$.

The following proposition can be useful  for estimating  $\hat
h_{\infty}$.

\begin{proposition}
\label{prop5.2}
Assume that $h,\gamma\in C^2_b(\BR^d)$. Then
\[
\hat h_\infty(x)\le R^D|Lh|(x)\le \big(\mathbb E_x\sup_{t\le\tau_D}|Lh(X_t)|^p\big)^{1/p}\big(\mathbb E_x(\tau_D)^{p^*}\big)^{1/p^*},\quad x\in D,\,p\in [1,\infty].
\]
\end{proposition}
\begin{proof}
By (\ref{Mtf}) applied to $h$ and $\gamma$,
\[
M^{\gamma}_t-M^h_t=(\gamma-h)(X_t)-(\gamma-h)(X_0)+\int_0^tLh(X_r)\,dr,\quad t\ge 0.
\]
By  (S3), for all $a\ge 0$ and $t\in[0,a]$ we have
\[
(\gamma-h)(X_t)=(\gamma-h)(X_{a\wedge\tau_D})
+\int_{t\wedge\tau_D}^{a\wedge \tau_D}Lh(X_r)\,dr
-\int_{t\wedge\tau_D}^{a\wedge \tau_D}\,d(M^{\gamma}_r-M^h_r).
\]
Thus the pair $((\gamma-h)(X),M^{\gamma}-M^h)$ is a solution of
BSDE$^{\tau_D}(0,F)$ with the coefficient $F(t,\omega):=
Lh(X_t(\omega))$, $t\ge 0$, $\omega\in\Omega$. Therefore applying
Proposition \ref{prop3.1} we get  the desired result.
\end{proof}

In the rest  of this subsection we assume additionally that
$\BM=\{(X,P_x),x\in\BR^d\}$ is a Feller process, i.e.
$P_t(C_{\infty}(\BR^d))\subset C_{\infty}(\BR^d)$, where $C_{\infty}(\BR^d)$ is the space of continuous functions on $\BR^d$ vanishing at infinity, with the symbol
\begin{align}
\label{eq5.symbf1}
p(x,\xi)&=c(x)-i\langle b(x),\xi\rangle+\frac12\langle Q
(x)\xi,\xi\rangle \nonumber \\
&\quad +\int_{\BR^d\setminus\{0\}}
\Big(1-e^{i\langle\xi,z\rangle}+i\langle\xi,z\rangle
\fch_{\{|z|\le 1\}}\Big)\,N(x,dz).
\end{align}

For an overview of  sufficient conditions on $p$ or the coefficients of $L$   guaranteeing that    $L$  generates a  Feller process $\mathbb M$ see \cite[Chapter 3]{BSW}. Here we recall one  general criterion (see \cite[Theorem 3.25, Lemma 3.26]{BSW}).
It says that if
\begin{enumerate}[(a)]
\item $\lim_{|x|\to\infty}N(x, B(-x,r))=0$ for any $r>0$,
\item $x\longmapsto p(x,\xi)$ is continuous for any $\xi\in\BR^d$,
\item the martingale problem for $L$ is well posed,
\end{enumerate}
then $\BM$ is Feller.
Using this  criterion and the results of \cite{St}
we get the following example of a Feller process.

\begin{example}
\label{ex5.7ab}
Assume that
\begin{equation}
\label{eq5.21} \lambda^{-1}I\le Q(x)\le\lambda I,\qquad
|b(x)|\le\Lambda,\quad 0\le c(x)\le \Lambda,\quad x\in\BR^d,
\end{equation}
for some $\lambda\ge1$, $\Lambda>0$  ($I$ is the $d$-dimensional
identity matrix),  $q_{ij}, b_i, c$ are  continuous, and $N_B$ (cf. condition (M2)) is continuous for any $B\in\mathcal B(\BR^d)$.
Furthermore, suppose that
\begin{equation}
\label{eq5.12mar.2}
\lim_{|x|\to\infty} \int_{\BR^d}\min\{|y|^2/|x|^2,1\}\,N(x,dy)=0.
\end{equation}
(for instance, the last condition is  satisfied if $N$ is independent of $x$
or  big jumps are integrable, that is  $\sup_{x\in\BR^d}\int_{\BR^d}\fch_{\{|y|\ge 1\}}\,N(x,dy)<\infty$).
Then $\mathbb M$ solving the martingale problem for $L$  is a Feller process.
Indeed, by \cite[Theorem 2.2, Theorem 4.3]{St} the martingale problem for $L$ is well posed, so we have (c).
By the assumptions made on the coefficients of $L$ and function $N_B$, (b) is satisfied. What is left is to show that (a) is satisfied. By \cite[Lemma 3.26]{BSW} condition (i) is satisfied if
\begin{equation}
\label{eq5.12mar.2c}
\lim_{|x|\to \infty} \sup_{|\xi|\le 1/|x|}(\mbox{\rm Re}\,p(x,\xi) -p(x,0)-\frac12 \xi\cdot Q(x)\xi)=0.
 \end{equation}
The quantity in the bracket equals to
\[
\int_{\BR^d}(1-\cos(\xi\cdot y))\,N(x,dy)\le \int_{\BR^d}\min\{|y|^2|\xi|^2,1\}\,N(x,dy).
\]
From this and \eqref{eq5.12mar.2} we easily get \eqref{eq5.12mar.2c}.
\end{example}

Consider the following condition
\begin{equation}
\label{eq5.18}
\lim_{|\xi|\to\infty}\frac{\inf_{z\in\BR^d}\mbox{Re}\, p(z,\xi)}{\log(1+|\xi|)}=\infty.
\end{equation}
Clearly, it is satisfied for the operator from Example
\ref{ex5.7ab}.  By \cite[Theorem
1.2]{SW}, if \eqref{eq5.18} holds, then $\BM$  has a transition
density $p(t,x,y)$. Consequently, the part $\BM^D$ of $\BM$ on $D$
has a transition density, which we denote by $p_D(t,x,y)$. Set
\begin{equation}
\label{eq5.dfofr284}
r(t)=(4\pi)^{-d}\int_{\BR^d}\exp\Big(-\frac{t}{16}\inf_{z\in\BR^d}\mbox{Re}\,p(z,\xi)\Big)\,d\xi.
\end{equation}
By \cite[Theorem 1.2]{SW}, under \eqref{eq5.18} we have
\begin{equation}
\label{eq5.12} r_{q^*}(t,x)\le \begin{cases} r(t), &\quad \mbox{if
}q^{*}=\infty,
\smallskip\\
r(t)[m(D)]^{1/{q^{*}}}, &\quad\mbox{if }q^{*}\in [1,\infty).
\end{cases}
\end{equation}

\begin{example}
\label{ex5.9}
Let $\alpha\in C^1_b(\BR^d)$, and
\[
0<\underline\alpha:=\inf_{x\in\BR^d}\alpha(x)\le\bar
\alpha:=\sup_{x\in\BR^d}\alpha(x)<2.
\]
Furthermore, suppose  that either $d\ge 2$,  or $d=1$ and there
exists $K>0$ such that $\sup_{|x|\ge K}\alpha(x)\in (0,1)$. Let
\[
p(x,\xi)=|\xi|^{\alpha(x)},\quad x,\xi\in\BR^d.
\]
Then, by \cite{SW}, there exists a Feller process $\mathbb M$ with
symbol $p$. One can  easily check that there is $C>0$ such that  for every $t\ge 1$,
\[
r(t)\le Ct^{-d/\bar\alpha}.
\]
The semigroup $(P_t)$ associated with $\mathbb M$ is generated by
the operator $\Delta^{\alpha(\cdot)}$.
\end{example}

In case $D$ is bounded one can estimate $r_1$ without imposing
condition (\ref{eq5.18}). Let $\delta=\mbox{diam}\,D$. If
$\delta<\infty$,  then by \cite[Theorem 5.9]{BSW},
\begin{equation}
\label{eq5.13}
r_1(t,x)\le 3e^{-tc(x,\delta)/16},\quad x\in D,
\end{equation}
where
\[
c(x,\delta):=\sup_{|\xi|\le1/(2\delta k^*(x,\delta))}\inf_{|y-x|\le3\delta}\mbox{Re}\,p(y,\xi),
\]
and
\[
k^*(x,\delta):=\inf\Big\{k\ge \big(\arcsin{\sqrt{2/3}}\big)^{-1}: \sup_{|\xi|\le 1/(2k\delta)}\sup_{|y-x|\le\delta}\frac{\mbox{Re}\,p(y,\xi)}{|\xi||\mbox{Im}\,p(y,\xi)|}\ge 4\delta\Big\}.
\]

\begin{remark}
Observe that if $p$ satisfies the following {\em sector condition}: there exists $\kappa>0$ such that
\[
|\mbox{\rm Im}\,p(x,\xi)|\le\kappa \mbox{\rm Re}\,p(x,\xi),\quad x,\xi\in\BR^d,
\]
then $k^*(x,\delta)= (\arcsin{\sqrt{2/3}}\big)^{-1},\,
x\in\BR^d,\, \delta>0$. The sector condition is trivially
satisfied when $\mbox{Im}\,p(\cdot,\cdot)=0$. In particular, the
sector  condition is satisfied by the symbol of the  operator
$\Delta^{\alpha(\cdot)}$ from Example \ref{ex5.9}.
\end{remark}

\begin{proposition}
\label{prop5.3ad}
\begin{enumerate}[\rm(i)]
\item If \mbox{\rm(\ref{eq5.18})} is satisfied and $w, \hat h_\infty\in L^1(D;m)$,  then
\begin{align*}
V^*_{T_s}(x)\le r(T_s)( \|w\|_{L^1(D;m)}+ \|\hat h_\infty\|_{L^1(D;m)}).
\end{align*}

\item Let $q\in (1,\infty]$. If \mbox{\rm(\ref{eq5.18})} is satisfied, $m(D)<\infty$ and $w, \hat h_\infty\in L^q(D;m)$,  then
\[
V^*_{T_s}(x) \le r(T_s)[m(D)]^{1/q^*}( \|w\|_{L^q(D;m)}
+ \|\hat h_\infty\|_{L^q(D;m)}).
\]

\item If $\delta:=\mbox{\rm diam}\,D<\infty$ and $w,\hat h_\infty\in L^\infty(D;m)$, then
\[
V^*_{T_s}(x) \le 3e^{-(T_s)c(x,\delta)/16}( \|w\|_{L^\infty(D;m)}
+ \|\hat h_\infty\|_{L^\infty(D;m)}).
\]
\end{enumerate}
\end{proposition}
\begin{proof}
Assertions (i), (ii) follow from Proposition \ref{prop5.5} and
(\ref{eq5.12}), and (iii) follows from  Proposition \ref{prop5.5}
and (\ref{eq5.13}).
\end{proof}

\subsection{Ultracontractivity and intrinsic ultracontractivity}
\label{sec5.3}

Let $m$ be a $\sigma$-finite measure on $\BR^d$ with full support
and $D\subset\BR^d$ be an open subset such that $m(D)<\infty$.
We assume that $\BM^D$ has the
transition density $p_D(t,x,y)$ with respect to $m$.
Let $(\hat P^D_t)$ be  the dual semigroup to
$(P^D_t)$ relative to $m$ in the sense that  for all $t>0$ and
nonnegative $f,g\in \BB(D)$,
\[
\int_{D} f(x) P^D_{t} g(x) m(d x)=\int_{D} g(x) \hat{P}^D_{t} f(x)
m(dx).
\]
%We also assume that  $\BM^D$ has  a transition density $p_D$ and
%there exists a Markov process $\hat{\BM}^D$ associated with the
%semigroup $(\hat{P}^D_t)$.
Clearly, $\hat P_tf(y)=\int_Dp_D(t,x,y)f(x)\,m(dx)$, $y\in D,\, f\in\BB^+(D)$.
The semigroup $(P^D_t)$ is said to be  {\em ultracontractive} if  for every  $t>0$ there exists $b(t)>0$
such that

\begin{equation}
\label{eq5.22}
p_D(t,x,y)\le b(t),\quad x,y\in D.
\end{equation}

It is known that if $(\hat P^D_t)$ is Markov, then $(P^D_t)$
is ultracontractive if and only if   the operators $P^D_t:L^2(D;m)\to L^\infty(D;m)$ and
$\hat P^D_t:L^2(D;m)\to L^\infty(D;m)$ are bounded.  Moreover, if $(P^D_t)$ is ultracontractive and $(\hat P^D_t)$ is  Markov, then \eqref{eq5.22}
holds with $b(t)=\max\{\|P^D_t\|_{L^2(D;m)\to L^\infty(D;m)}, \|\hat P^D_t\|_{L^2(D;m)\to L^\infty(D;m)}\}$ and $b(t)$ is nonincreasing as $t\to \infty$ (see, e.g., \cite[Proposition 2.2]{KS}). %It is an elementary check that
%$b(t)$ is nonincreasing as $t\to \infty$.

\begin{remark}
Let $\BM$ be a  Feller process with symbol \eqref{eq5.symbf1}. By \cite[Theorem 1.2]{SW}, if (\ref{eq5.18}) is satisfied, then $(P^D_t)$ is ultracontractive and
\begin{equation}
\label{eq5.dest65}
p_D(t,x,y)\le r(t),\quad t>0,\, x,y\in D,
\end{equation}
where $r$ is defined by (\ref{eq5.dfofr284}).
\end{remark}

\begin{proposition}
\label{prop5.11} Let $q\in [1,\infty)$. If $w,\hat h_\infty\in
L^q(D;m)$ and $(P^D_t)$ is ultracontractive, then
\begin{equation}
\label{eq5.impr1}
V^*_{T_s}(x)\le\big(r_1(T_s,x)\big)^{1/q^*}\big(b(T_s)\big)^{1/q}
\big( \|w\|_{L^q(D;m)}+ \|\hat h_\infty\|_{L^q(D;m)}\big).
\end{equation}
\end{proposition}
\begin{proof}
Let $\rho\in L^q(D;m)$. Then
\begin{align*}
P^D_t\rho=\mathbb
E_x\fch_{\{t<\tau_D\}}\rho(X_t)&
=\BE_x\fch_{\{t<\tau_D\}}\big(\rho(X_t)\fch_{\{t<\tau_D\}}\big)
\\ &\le \big(P_x(t<\tau_D)\big)^{1/q^*}\big(\BE_x\fch_{\{t<\tau_D\}}\rho^q(X_t)\big)^{1/q}
\\ &=\big(r_1(t,x)\big)^{1/q^*}\Big(\int_Dp_D(t,x,y)\rho^q(y)\,m(dy)\Big)^{1/q}
\\ &\le \big(r_1(T_s,x)\big)^{1/q^*}\big(b(T_s)\big)^{1/q} \|\rho\|_{L^q(D;m)}.
\end{align*}
From this and Theorem \ref{th5.1} we get the desired estimate.
\end{proof}

\begin{remark}
(i) It is worth noting that the difference between
\eqref{eq5.impr1} and  the inequality formulated  in Proposition \ref{prop5.3ad}(ii)
is that in \eqref{eq5.impr1} the factor $r(T_s)[m(D)]^{1/q^*}$ appearing in Proposition \ref{prop5.3ad}(ii) has been replaced by $\big(r_1(T_s,x)\big)^{1/q^*}\big(b(T_s)\big)^{1/q}$.
This is an important refinement.  Recall that
\[
r_1(T_s,x)=P_x(\tau_D>T_s)=\int_Dp_D(T_s,x,y)\,m(dy),
\]
while  $r(T_s)=\sup_{x,y\in\BR^d}p(T_s,x,y)$.
The asymptotics   of the  quantity $P_x(\tau_D>T_s)$ as $T_s\to\infty$ is  well studied in the literature.
Moreover, the term $r_1(T_s,x)$ tends to zero when $x$ approaches the boundary of $D$, while
$r(T_s)$ is independent of $x\in D$.
\smallskip\\
(ii) The decay of $V^*_{T_s}$ similar to \eqref{eq5.impr1}  is asserted in Proposition \ref{prop5.3ad}(iii). However, in Proposition \ref{prop5.3ad}(iii)
boundedness of  $w$ and  $\hat h_\infty$ is required.  Thanks to the ultracontractivity of $(P^D_t)$ we may dispense with this restriction.
\end{remark}

\begin{corollary}
\label{cor5.4} Let $q\in [1,\infty)$.  Assume that
$\delta=\mbox{\rm diam}\,D<\infty$, $w,\hat h_\infty\in L^q(D;m)$,
and $\mathbb M=\{(X,P_x),x\in\BR^d\}$ is a Feller process with the symbol $p$ $($cf. \eqref{eq5.symbf1}$)$.  If
$(P^D_t)$ is ultracontractive, then
\[
V^*_{T_s}(x)\le
3^{1/q^*}e^{-T_sc(x,\delta)/(16q^*)}\big(b(T_s)\big)^{1/q}
\big( \|w\|_{L^q(D;m)}+ \|\hat h_\infty\|_{L^q(D;m)}\big).
\]
\end{corollary}
\begin{proof}
Follows from Proposition \ref{prop5.11} and \eqref{eq5.13}.
\end{proof}

%\begin{remark}
%\label{rem5.17}
%Suppose  we know that $p_D$ satisfies (\ref{eq5.22}) with $b$ replaced by
%some positive function $\beta:(0,\infty)\rightarrow\BR$.
%The proof of Proposition \ref{prop5.11} shows that then (\ref{eq5.impr1}) holds with $b$ replaced by $\beta$
%(we need not know that $(P^D_t)$ is ultracontractive).
%\end{remark}

\begin{example}[divergence form operators]
\label{ex5.65}
Let $Q, b$ and $c$ be as in \eqref{eq.opa} and   satisfy \eqref{eq5.21}.
By \cite{St2} there exists a diffusion process $\BM$ with transition density $p(\cdot,\cdot,\cdot)$
being a fundamental solution for the operator
\begin{equation}
\label{eq5.dfop98}
L=\sum^d_{i,j=1}\partial_{x_i}(q_{ij}(x)\partial_{x_j})+\sum_{i=1}^db_i\partial_{x_i}+c.
\end{equation}
By Aronson's estimates, there exists $M>0$ depending only on $\Lambda, d$ and $T>0$
such that
\[
p(t,x,y)\le Mt^{-d/2}e^{-|x-y|^2/(Mt)}
\]
for all $t\in (0,T]$ and $x,y\in\BR^d$. By Proposition \ref{prop5.11}, (\ref{eq5.impr1}) holds with $b$ defined by the right-hand side of the above inequality. Better rate of convergence we get in case $c=0$, $b^i=0$, $i=1,\dots,d$, and  $D$ is Lipschitz (see Example \ref{ex5.21}).
\end{example}

In the rest of this subsection we assume that for each $t>0$,
$p_D(t,\cdot,\cdot)$ is   bounded and strictly
positive. Moreover, we assume that $(\hat P^D_t)$ is Markov. Let
\begin{equation}
\label{eq5.23}
-\lambda_1=\sup\{\mbox{Re}\lambda:\lambda\in\sigma(L^D)\},
\end{equation}
where $\sigma(L^D)$ is the spectrum  of the infinitesimal
generator of the semigroup $(P^D_t)$ on $L^2(D;m)$. By Jentzsch's
theorem (see \cite[Theorem V.6.6, page 337]{Schaefer}) there exist
unique, up to a multiplicity constant, strictly positive functions
$\phi_1,\hat\phi_1\in L^2(D;m)$  such that
\[
P^D_t\phi_1(x)=e^{-t\lambda_1}\phi_1(x),
\qquad\hat P^D_t\hat\phi_1(x)=e^{-t\lambda_1}
\hat \phi_1(x),\quad x\in D,\, t>0,
\]
The semigroup $(P^D_t)$ is called {\em intrinsically
ultracontractive} (see \cite{DaS,KS}) if for every $t>0$ there exists a constant
$c_t>0$ such that
\[
p_D(t,x,y)\le c_t\phi_1(x)\hat\phi_1(y),\quad x,y\in D.
\]
Equivalently, $(P^D_t)$ is  intrinsically  ultracontractive if
$(Q^D_t)$ is ultracontractive on $L^2(D;\mu)$, where
\[
Q^D_tf(x)=e^{\lambda_1t}P^D_t(f\phi_1)\phi_1^{-1},
\qquad \mu(dy)=\phi_1(y)\hat\phi_1(y)\,m(dy).
\]
 From the last statement  and the fact that $(Q^D_t)$ $(\hat Q^D_t)$ are Markov, we infer, in
particular, that for every $t_0>0$ there exists $M(t_0)>0$ such that
\begin{equation}
\label{eq5.19} p_D(t,x,y)\le
M(t_0)e^{-\lambda_1t}\phi_1(x)\hat\phi_1(y), \quad x,y\in D,\,
t\ge t_0,
\end{equation}
with $M(t_0)=\max\{\|Q^D_t\|_{L^2(D;\mu)\to L^\infty(D;\mu)}, \|\hat Q^D_t\|_{L^2(D;\mu)\to L^\infty(D;\mu)}\}$.

\begin{proposition}
\label{prop5.14}
Let $q\in [1,\infty]$  and $t_0>0$. Assume that
$(P^D_t)$ is intrinsically ultracontractive and $w,\hat
h_\infty\in L^q(D;m)$. Then for all $T_s\ge t_0$ and $x\in D$,
\[
V^*_{T_s}(x) \le M(t_0)e^{-\lambda_1 T_s}\phi_1(x)
\|\hat\phi_1\|_{L^{q^*}(D;m)}(\|w\|_{L^q(D;m)} + \|\hat
h_\infty\|_{L^q(D;m)}).
\]
\end{proposition}
\begin{proof}
Let $t\ge t_0$. By (\ref{eq5.19}),
\[
r_{q^*}(t,x)\le  M(t_0)e^{-\lambda_1 t}\phi_1(x)
\|\hat\phi_1\|_{L^{q^*}(D;m)},
\]
so the desired estimate follows from  Proposition \ref{prop5.5}.
\end{proof}

\begin{example}[nondivergence form operator]
\label{ex5.15} Let $m$ be the Lebesgue measure on $\BR^d$ and
$D\subset\BR^d$ be a bounded domain. Consider the operator
(\ref{eq.opa}) with $N=0$. Assume that $q_{ij},b_i$,
$i,j=1,\dots,d$, and $c$ are bounded $C^{\infty}$ functions on
$\BR^d$. Furthermore, we assume that (\ref{eq5.21}) is satisfied,
$\partial b_i/\partial x_i$, $i=1,\dots,d$, are bounded and
\[
c(x)-\sum^d_{i=1}\partial_{x_i}b_i(x)\ge0, \quad x\in \BR^d.
\]
Then the assumptions  formulated in the first paragraph of this subsection
are satisfied (see \cite[p. 538]{KS}). If $D$ is
Lipschitz, then by \cite[Theorem 3.9]{KS} the semigroup $(P^D_t)$
is intrinsically ultracontractive.
\end{example}

\begin{example}[divergence form operator]
\label{ex5.21}
Consider the operator $L$ from Example \ref{ex5.65} but assume that $b^i=0$, $i=1,\dots,d$, and $c=0$. In \cite{DaS} it is proved that if $D$ is bounded and Lipschitz, then $(P^D_t)$ is intrinsically ultracontractive. Thus,  Proposition \ref{prop5.14}  is applicable.
This proposition requires $\hat h_\infty$ to be in $ L^q(D;m)$. In Subsection 5.4 (see Example \ref{ex5.28}) we provide different results on the rate of convergence  requiring assumptions on $h$ and not on $\hat h_\infty$.
\end{example}

\begin{example}
Let $\BM=(X,P_x)$ be a symmetric L\'evy process in $\BR^d$ with L\'evy measure $\nu$ such that
\begin{equation}
\label{eq5.24}
\nu(B(x,r))>0
\end{equation}
for all $x\in\BR^d$ and $r>0$, where $B(x,r)=\{y\in\BR^d:|x-y|<r\}$.
Assume also that $\BM$ has
the transition density $p(t,x,y)=p(t,x-y)$ (with respect to the Lebesgue measure)  such that
$p(t,\cdot,\cdot)$ is continuous for every $t>0$ and moreover,  for every $\delta>0$ there exists a constant $c(\delta)>0$ such that $p(t,x)\le c(\delta)$ for all $t>0$ and $|x|\ge\delta$.
In \cite[Theorem 3.1]{G} (see also \cite{Ku} for the special case of rotationally symmetric $\alpha$-stable L\'evy process) it is proved that for any bounded open set $D\subset\BR^d$ the semigroup $(P^D_t)$ associated with the process $\BM$ killed upon exiting $D$ is intrinsically ultracontractive. Note also that (\ref{eq5.24}) can be weaken if we additionally assume that $D$ is a connected Lipschitz set.
\end{example}

\subsection{Dirichlet forms}
\label{sec5.4}

Let $E$ be a locally compact separable metric space and $m$  be a
Radon measure on $E$ with full support. Let $(\EE,\DD(\EE))$ be a
regular symmetric Dirichlet form on $L^2(E;m)$. For an open
$U\subset E$ we define the capacity of $U$ by
\[
\mbox{Cap}_E(U)=\inf\{\EE_1(u,u): u\ge\fch_{U}\,\, m\mbox{-a.e., }
u\in \DD(\EE)\},
\]
where $\EE_1(u,u)=\EE(u,u)+\|u\|^2_{L^2(E;m)}$. For an arbitrary
$B\subset D$ we set
\[
\mbox{Cap}_E(B)=\inf_{B\subset U}\mbox{Cap}_E(U).
\]
Recall that a function $u$ on $E$ is called quasi-continuous if
for every $\varepsilon>0$ there exists an open set
$U_\varepsilon\subset E$ such that
$\mbox{Cap}_E(U_\varepsilon)<\varepsilon$ and $u|_{E\setminus
U_\varepsilon}$ is continuous. By \cite[Theorem 2.1.3]{FOT} every
function $u\in \mathcal D(\EE)$ has an  $m$-version $\tilde u$
which is  quasi-continuous. Let  $(\EE^D,\mathcal D(\EE^D))$ be
the part of $(\EE,\mathcal D(\EE))$ on $D$, i.e.
\[
\EE^D(u,v)=\EE(u,v),\quad u,v\in \mathcal D(\EE^D) :=\{u\in
\mathcal D(\EE): \tilde u=0\mbox{ q.e. on }E\setminus D\}
\]
(here $\tilde u=0$ q.e. means that $\tilde u=0$ except for a set
of capacity $\mbox{Cap}_E$ equal to zero). By \cite[Section
7]{FOT} there exists a Hunt process $\BM$ associated with
$(\EE,\mathcal D(\EE))$. Moreover, $\BM^D$ is associated with
$(\EE^D,\mathcal D(\EE^D))$ which is again a regular Dirichlet
form  (see \cite[Theorem 4.4.3]{FOT}). Analogously to
$\mbox{Cap}_E$ we define $\mbox{Cap}_D$. By \cite[Theorem
2.1.3]{FOT}, any function $u\in \DD(\EE^D)$ has an $m$-version
$\tilde u$  which is quasi-continuous. Formally we should write
"quasi-continuous with respect to $\mbox{Cap}_E$ or
$\mbox{Cap}_D$" and not just "quasi-continuous". However both
capacities are  equivalent on $D$ (see \cite[Theorem
4.4.3]{FOT}), so  the above terminology does not lead to ambiguity.

Throughout the subsection, we assume that $(\EE^D,\DD(\EE^D))$ is
transient, i.e. there exists a strictly positive $\chi\in
L^2(D;m)$ such that
\begin{equation}
\label{eq5.144}
\int_D |u|\chi\,dm\le (\EE^D(u,u))^{1/2},\quad u\in\DD(\EE^D).
\end{equation}
By \cite[Theorems 1.5.2, 1.5.3]{FOT}, there exists an extension
$\mathcal D_e(\EE^D)\subset L^1(D;\nu\cdot m)$ of $\mathcal
D(\EE^D)$, called the extended Dirichlet space, such that
$(\EE^D,\mathcal D_e(\EE^D))$ is a Hilbert space. Furthermore,
$\mathcal D_e(\EE^D)\cap L^2(D;m)=\mathcal D(\EE^D)$ is dense in
$\mathcal D_e(\EE^D)$. By \cite[Theorem 2.1.3]{FOT} (see also the
comments following \cite[(2.1.14)]{FOT}) every function  $u\in
\mathcal D_e(\EE^D)$ has an  $m$-version $\tilde u$  which is
quasi-continuous. From now on, we  consider quasi-continuous
versions of functions  in $\DD(\EE^D)$ or $\DD_e(\EE^D)$.

Below we first give estimates of $V^*_{T_s}$ in the space
$L^1(D;\rho\cdot m)$ with some weight $\rho$. Then we give
pointwise estimates under  additional assumptions on $\BM^D$.

\begin{lemma}
\label{lm5.1}
Assume that $u\in \mathcal D_e(\EE^D)$. Then for every positive $\rho\in L^2(D;m)$ such that $R^D\rho\in L^2(D;m)$,  and every $t\ge0$,
\begin{equation}
\label{eq5.146}
\int_D\big(\sup_{\tau\in\TT^{\tau_D}}\mathbb E_x |u(X_\tau)|\big)\,\rho(x)\,m(dx)
\le(\EE^D(u,u))^{1/2}\|\rho\|^{1/2}_{L^2(D;m)}\|R^D\rho\|^{1/2}_{L^2(D;m)}.
\end{equation}
\end{lemma}
\begin{proof}
We let
\begin{equation}
\label{eq5.dfofga1}
\Gamma=\{\eta\in \DD_e(\EE):\eta\ge|u|\,\, m\mbox{-a.e.}\}.
\end{equation}
By \cite[Theorem 1.1.1]{Oshima} applied to $\Gamma$ and $J=0$ there exists a unique $v\in \Gamma$  such that $\EE(v,w-v)\ge0$ for all $w\in\Gamma$. Moreover, $v$ has the property that
\[
\EE^D(v,v)=\inf\{\EE^D(\phi,\phi): \phi\in \DD_e(\EE^D)\,\,\mbox{and }\phi\ge |u|\,\, m\mbox{-a.e.} \}.
\]
Let $\bar w$ be a positive element of $\DD(\EE^D)$. Then $\bar w+v\in\Gamma$, so $\EE(v,\bar w)=\EE(v,\bar w+v-v)\ge0$. Consequently, $v$ is excessive by
\cite[Theorem 1.4.1]{Oshima}. From this and  \cite[Theorem A.2.5]{FOT} it follows that $v(X)$ is a c\`adl\`ag process. Furthermore, from the fact that $v$ is excessive and the Markov property it follows that $v(X)$ is a supermartingale  under the measure $P_x$ for $m$-a.e. $x\in D$. Hence
\[
\sup_{\tau\in\TT^{\tau_D}}\BE_x
|u(X_\tau)|\le\sup_{\tau\in\TT^{\tau_D}}\BE_x v(X_\tau)\le \BE_x
v(X_0)=v(x)
\]
for $m$-a.e. $x\in D$. Multiplying both sides of the above
inequality by $\rho$ and then integrating with respect to $m$
yields
\begin{equation}
\label{eq5.012} \int_D\big(\sup_{\tau\in\TT^{\tau_D}}\mathbb E_x
|u(X_\tau)|\big)\,\rho(x)\,m(dx) \le \int_D v \cdot
\rho\,dm=\EE^D(v,R^D\rho).
\end{equation}
On the other hand,
\begin{align}
\label{eq5.013} \EE^D(v,R^D\rho)&\le (\EE^D(v,v))^{1/2}
(\EE^D(R^D\rho,R^D\rho))^{1/2} \nonumber\\
&=(\EE^D(v,v))^{1/2}
(\rho,R^D\rho)^{1/2}_{L^2(D;m)}\nonumber \\
& \le (\EE^D(|u|,|u|))^{1/2}\|\rho\|^{1/2}_{L^2(D;m)}
\|R^D\rho\|^{1/2}_{L^2(D;m)}
\end{align}
Since $\EE^D$ is a transient Dirichlet form, every normal
contraction operates on $\EE^D$. Hence $|u|\in\DD(\EE^D)$ and
$\EE^D(|u|,|u|)\le \EE^D(u,u)$. Therefore (\ref{eq5.012}) and
(\ref{eq5.013}) imply  (\ref{eq5.146}).
\end{proof}

By  \cite[Theorem 4.3.2]{FOT}, if $\psi\in\DD_e(\EE)$, then $\gamma\in\DD_e(\EE)$.
Below we shall freely  use this fact without explicit mention.

\begin{lemma}
\label{lem5.3} Assume that $h,\psi\in\DD_e(\EE)$. Then for all
$t>0$ and positive $\rho\in L^2(D;m)$ such that $R^D\rho\in
L^2(D;m)$,
 \[
\int_DP^D_t\hat h_\infty(x)\,\rho(x)\,m(dx)
\le(\EE^D(\gamma-h,\gamma-h))^{1/2}\|P^D_{t}
\rho\|^{1/2}_{L^2(D;m)}\|R^D\rho\|^{1/2}_{L^2(D;m)}.
\]
\end{lemma}
\begin{proof}
Since  $\gamma\in \DD_e(\EE)$ and
$h=\gamma$ on $E\setminus D$, we have  $\gamma-h\in \DD_e(\EE^D)$.
Therefore the result follows from the definition of $\hat h_{\infty}$
and Lemma \ref{lm5.1} applied to $u=\gamma-h$.
\end{proof}

We are now ready to give an estimate for $V^*_{T_s}$ in case
$h,\psi\in\DD_e(\EE)$. It is worth stressing that in general,
under the assumption  $h\in\DD_e(\EE)$ the process $h(X)$ is not a
semimartingale under $P_x$.  Therefore Lemma \ref{lem5.3} and
Proposition \ref{prop5.18} below apply in situations quite
different from those considered earlier, for instance in Corollary
\ref{cor5.2} or Proposition \ref{prop5.2}.

\begin{proposition}
\label{prop5.18} Assume that  $w\in L^2(D;m)$, $\psi, h\in
\mathcal D_e(\EE)$ and $\rho\in L^2(D;m)$ is a positive
function such that $R^D\rho\in L^2(D;m)$. Then
\begin{align*}
&\int_DV^*_{T_s}(x)\rho(x)\,m(dx)\\
&\qquad\le \Big(\|w\|_{L^2(D;m)}+
(\EE^D(\gamma-h,\gamma-h))^{1/2}\|R^D\rho\|^{1/2}_{L^2(D;m)}\Big)
\|P^D_{T_s}\rho\|^{1/2}_{L^2(D;m)}.
\end{align*}
\end{proposition}
\begin{proof}
Follows from Lemma \ref{lem5.3} and Theorem \ref{th5.1}.
\end{proof}

Let $\alpha:(0,\infty)\to(0,\infty)$ be a nonincreasing function
and $\Phi:L^2(D;m)\to[0,\infty]$ satisfy the following conditions:
\begin{enumerate}
\item[(a)] $\Phi(cu)=c^2\Phi(u)$ for all $ u\in L^2(D;m)$ and $c\in\BR$,
\item[(b)] $\Phi(P^D_t\rho)\le\Phi(\rho)$ for all $\rho\in L^2(D;m)$.
\end{enumerate}
Suppose that the following Sobolev-type inequality is satisfied:
\begin{equation}
\label{eq5.sob}
\|u\|^2_{L^2(D;m)}\le\alpha(r)\EE^D(u,u)+r\Phi(u),
\quad r>0,\, u\in \mathcal D(\EE^D).
\end{equation}
By \cite{Wang}, if \eqref{eq5.sob} is satisfied,  then
\begin{equation}
\label{eq5.149} \|P^D_t\rho\|^2_{L^2(D;m)}\le
\beta(t)(\Phi(\rho)+\|\rho\|^2_{L^2(D;m)}),\quad t>0,\, \rho\in
L^2(D;m),
\end{equation}
where
\[
\beta(t)=\inf\{r>0: -\alpha(r)\ln r\le 2t \},\quad t>0.
\]

\begin{remark}
\label{rem5.19} Let $\lambda_2$ be the bottom of the spectrum,
that is
\begin{equation}
\label{eq5.20}
\lambda_2=\inf\{\EE^D(u,u): u\in \DD(\EE^D),\, \|u\|_{L^2(D;m)}=1\}.
\end{equation}
If $\lambda_2>0$, then by \eqref{eq5.149},
\[
\|P^D_t \rho\|_{L^2(D;m)}\le e^{-2\lambda_2 t}\|\rho\|_{L^2(D;m)},
\quad \rho\in L^2(D;m).
\]
Indeed, for $u\in\DD_e(\EE^D)$ we have $\EE^D(u,u)\le
\lambda_2^{-1}\|u\|^2_{L^2(D;m)}$, so (\ref{eq5.sob}) is satisfied
with $\alpha(r)=\lambda^{-1}_2$, $r>0$, $\Phi=0$.  Clearly, we
then have $\beta(t)=e^{-2t\lambda_2}$, and we apply
(\ref{eq5.149}).
\end{remark}

\begin{corollary}
\label{cor5.6} Let $w,\psi, h$ and $\rho$ satisfy the assumptions
of Proposition \ref{prop5.18}. If \eqref{eq5.sob} is satisfied,
then
\begin{align*}
\int_DV^*_{T_s}(x)\rho(x)\,m(dx)\le C(\beta(T_s))^{1/4},
\end{align*}
where
\[
C=\Big(\|w\|_{L^2(D;m)}+
(\EE^D(\gamma-h,\gamma-h))^{1/2}\|R^D\rho\|^{1/2}_{L^2(D;m)}\Big)
\big(\Phi(\rho)+\|\rho\|^2_{L^2(D;m)}\big)^{1/4}.
\]
\end{corollary}
\begin{proof}
Follows from Proposition \ref{prop5.18} and \eqref{eq5.149}.
\end{proof}
%\begin{equation}
%\label{eq5.145}
%\int_D u^2\nu\,dm\le \EE^D(u,u),\quad u\in D(\EE^D).
%\end{equation}

One can improve slightly the rate of convergence given above  provided we
know that $\Phi(\hat h_\infty)<\infty$.

\begin{proposition}
\label{prop5.4} Assume that   $w\in L^2(D;m)$, $\psi, h\in
\DD_e(\EE)$, and \eqref{eq5.sob} is satisfied. Then for every
positive  $\rho\in L^2(D;m)$, and any $r>0$,
\begin{align*}
\int_DV^*_{T_s}(x)\rho(x)\,m(dx)\le C_r(\beta(T_s))^{1/2},
\end{align*}
where
\[
C_r=\big(\|w\|_{L^2(D;m)}+ (\alpha(r)\EE^D(\gamma-h,\gamma-h)
+r\Phi(\hat h_\infty))^{1/2}\big)\big(\|\rho\|^2_{L^2(D;m)}
+\Phi(\rho)\big)^{1/2}.
\]
\end{proposition}
\begin{proof}
First observe that if we knew that \eqref{eq5.146} holds with the
right-hand side replaced by $(\alpha(r)\EE^D(u,u)+r\Phi(\hat
h_\infty))^{1/2}\|\rho\|_{L^2(D;m)}$, then using this modified
inequality and repeating step by step the reasoning of the proof
of Lemma \ref{lem5.3} and Proposition \ref{prop5.18} we would get
\begin{align*}
&\int_DV^*_{T_s}(x)\rho(x)\,m(dx)\le C_r \|P^D_{T_s}\rho\|_{L^2(D;m)}
\end{align*}
with $C_r$ as in the proposition. This together with \eqref{eq5.149} implies the desired estimate. Therefore, what is left is to show that for every $u\in \mathcal D_e(\EE^D)$,
\begin{equation}
\label{eq5.148}
\int_D\big(\sup_{\tau\in\TT^{\tau_D}}\mathbb E_x |u(X_\tau)|\big)\,\rho(x)\,m(dx)\le(\alpha(r)\EE^D(u,u)+r\Phi(\hat h_\infty))^{1/2}\|\rho\|_{L^2(D;m)}.
\end{equation}
Let $v$ be defined as in the proof of Lemma \ref{lm5.1}.
The proof of \eqref{eq5.148} differs from the proof of \eqref{eq5.146} only in the estimate of the term $\int_D v\rho\,dm$ appearing in (\ref{eq5.012}).
In the present situation, by Schwartz's inequality and \eqref{eq5.sob} we have
\begin{align*}
\int_Dv\rho\,dm\le \|v\|_{L^2(D;m)}\|\rho\|_{L^2(D;m)}&\le (\alpha(r)\EE^D(v,v)+r\Phi(v))^{1/2}\|\rho\|_{L^2(D;m)}\\&
\le  (\alpha(r)\EE^D(|u|,|u|)+r\Phi(v))^{1/2}\|\rho\|_{L^2(D;m)}\\&
\le  (\alpha(r)\EE^D(u,u)+r\Phi(v))^{1/2}\|\rho\|_{L^2(D;m)}.
\end{align*}
Take $u=\gamma-h$ in \eqref{eq5.dfofga1}. Then, by \cite[Proposition 3.16, Lemma 3.7]{K:SM}, the process  $v(X)$
is the first component of the    unique solution of RBSDE$^{\tau_D}(0,0,|\gamma-h|(X))$.
Hence, by Theorem  \ref{th3.3},    $v=\hat h_\infty$, and the proof is complete.
\end{proof}

\begin{corollary}
\label{cor5.fidode5}
Suppose that the assumptions of Proposition \ref{prop5.4} hold, and moreover there exists $\varepsilon>0$ such that $\Phi(\eta)\le \varepsilon \|\eta\|^2_{L^2(D;m)}$, $\eta\in L^2(D;m)$. Then the conclusion of Proposition \ref{prop5.4} holds with $C_r$ replaced by
\[
\hat C:= (1+\varepsilon)^{1/2}\big(\|w\|_{L^2(D;m)}
+(2\alpha(1/(2\varepsilon))\EE^D(\gamma-h,\gamma-h))^{1/2}\big)\|\rho\|_{L^2(D;m)}.
\]
\end{corollary}
\begin{proof}
Let $u=\gamma-h\in\DD_e(\EE^D)$ and $v$ be defined as in the proof  of Lemma \ref{lm5.1}. In the proof of Proposition \ref{prop5.4} it is shown that then $v=\hat h_{\infty}$.
Hence
\[
\EE^D(\hat h_\infty,\hat h_\infty)\le \EE^D(|\gamma-h|,|\gamma-h|)\le \EE^D(\gamma-h,\gamma-h).
\]
The desired result follows from Proposition \ref{prop5.4}, the assumption on $\Phi$, and the above inequality.
\end{proof}

\begin{example}
\label{ex5.28}
Let $D$ be a domain in $\BR^d$ bounded at least in one direction, and $m$ be the Lebesgue measure on
$\BR^d$. Let $q_{ij}:D\rightarrow\BR$, $i,j=1,\dots,d$, be
measurable functions such   for each $x\in D$ the matrix
$Q(x)=[q_{ij}(x)]^d_{i,j=1}$ is  symmetric and satisfies the first
condition in (\ref{eq5.21}). Consider the form
\[
\EE^D(u,v)=\int_DQ(x)\nabla u(x)\nabla v(x)\,dx,\quad u,v\in
\DD(\EE^D)=H^1_0(D).
\]
It is well known (see, e.g., \cite[Section 3.1]{FOT}) that
$(\EE^D,H^1_0(D))$ is a symmetric regular Dirichlet form on
$L^2(D;m)$. It is transient by (\ref{eq5.21}) and Poincar\'e's
inequality (see \cite[Example 1.5.1]{FOT}). Its generator is the
divergence form operator (cf. \eqref{eq5.dfop98})
\[
L=\sum^d_{i,j=1}\partial_{x_i}(q_{ij}(x)\partial_{x_j}).
\]
Define $\lambda_2$ by (\ref{eq5.20}). By Poincar\'e's inequality $\lambda_2>0$, so by Remark
\ref{rem5.19}, Corollary \ref{cor5.fidode5} applies with
$\beta(t)=e^{-2t\lambda_2}$, $t>0$. Note that  in general, even
for regular $h$, the process $h(X)$ is not a semimartingale under
$P_x$.  It is known that if  $h\in H^1_0(D)$, then for $m$-a.e.
$x\in D$, under the measure $P_x$,   it is a Dirichlet process  in
the sense of F\"ollmer (for details and refinements see
\cite{R:PA}).
\end{example}

To get pointwise estimates for $V^*_{T_s}$, in the rest of this
subsection we assume that $D$ is connected and $\BM^D$ is strongly
Feller, that is $R^D_1(\BB_b(D))\subset C_b(D)$, and   $R^D_11\in C_{\infty}(D)$.
Therefore assumptions I--III of \cite[Section 6.4]{FOT} are
satisfied.  Recall that  $\lambda_2$  is defined by (\ref{eq5.20}).
By \cite[Theorem 6.4.4]{FOT}, for every $\lambda<\lambda_2$ we have
\[
c_\lambda:=\sup_{x\in D}\mathbb E_xe^{\lambda\tau_D}<\infty.
\]

\begin{proposition}
Assume that $w,\hat h_\infty\in L^\infty(D;m)$.
Then for all $\lambda<\lambda_2$ and $x\in D$,
\begin{align*}
V^*_{T_s}(x)&\le c_\lambda e^{-\lambda T_s}(\|w\|_{L^\infty(D;m)}+
\|\hat h_\infty\|_{L^\infty(D;m)})\\&\le  c_\lambda e^{-\lambda T_s}(\|w\|_{L^\infty(D;m)}+
\|h-\gamma\|_{L^\infty(D;m)}).
\end{align*}
\end{proposition}
\begin{proof}
Let $\rho$ be a positive function in $L^\infty(D;m)$. Then for any $t>0$,
\begin{align*}
P^D_t\rho(x)=\BE_x\fch_{\{t<\tau_D\}}\rho(X_t)& \le
P_x(t<\tau_D)\|\rho\|_{L^\infty(D;m)}
=P_x(e^t<e^{\tau_D})\|\rho\|_{L^\infty(D;m)}\\
&\le e^{-\lambda t}
\BE_xe^{\lambda\tau_D}\|\rho\|_{L^\infty(D;m)}\le e^{-\lambda t}
c_\lambda\|\rho\|_{L^\infty(D;m)}.
\end{align*}
The desired result now follows from Theorem \ref{th5.1}.
\end{proof}

\subsection{The semigroup theory approach}
\label{sec5.5}

Let $m$ be a $\sigma$-finite positive Borel measure on $E$. In
what follows $Y$ denotes the Banach space $C_{\infty}(D)$ or $L^p(D;m)$
with $p\ge 1$. By $\|\cdot\|_Y$ we denote the natural norm on $Y$.
Assume that $(P^D_t)$ is a $C_{0}$-semigroup on $Y$. By
\cite[Theorem 4.4.1]{Pazy}, if for some $q\in[1,\infty)$,
\[
\int_0^\infty\|P^D_tf\|^q_Y\,dt<\infty,\quad f\in Y,
\]
then there are constants  $b\ge 1$ and $\lambda>0$ such
that
\begin{equation}
\label{eq5.14} \|P^D_tf\|_Y\le be^{-\lambda t}\|f\|_Y,\quad t>0.
\end{equation}
This estimate and Theorem \ref{th5.1} can be used to estimate
$V^*_{T_s}$ for some specific $w$ and $\hat h_{\infty}$.

The assumption that $(P^D_t)$ is a $C_0$-semigroup on $L^p(D;m)$
with $p\ge1$  is satisfied for instance  if $m$ is excessive, that
is
\[
\int_E P_tf\,dm\le \int_Ef\,dm,\quad f\in\BB^+(D).
\]
On the other hand, if $\BM$ is  Feller, then $(P^D_t)$ is a
$C_0$-semigroup of contractions on $C_{\infty}(D)$ (since $D$ is assumed
to be Dirichlet regular).

The other useful result which can be applied in our context  says
(see, e.g., \cite[Theorem 4.4.3]{Pazy}) that if $(P^D_t)$ is
analytic and $-\lambda_1<0$, where $-\lambda_1$ is defined by (\ref{eq5.23}), then
there are constants $b\ge 1$ and $\lambda>0$ such that
(\ref{eq5.14}) is satisfied. Note also that if $\BM^D$ is
$m$-symmetric, i.e. $\int_E P^D_tf\cdot g\,dm=\int_E f\cdot  P^D_t
g\,dm$ for $f,g\in\BB^+(E)$, then $(P^D_t)$ is analytic on
$L^p(D;m)$ for any $p>1$ (see \cite{LP}).

For analyticity of
$(P^D_t)$ on $C_{\infty}(D)$ see, e.g., \cite[Section
1.2.2]{Taira} and \cite[Section 3]{Cerrai}.
If $(P^D_t)$ is ultracontractive and  is  associated with a symmetric Dirichlet form $(\EE,D(\EE))$, then
by \cite[Theorem 2.1.5]{DaS}, $(P^D_t)$ is analytic on $L^\infty(D)$.
Finally,  note that in  \cite{GT} it is proved that if $(P^D_t)$ is analytic on a Banach space $Y$, then for any Bernstein function $\psi:(0,\infty)\to [0,\infty)$ the subordinated semigroup
\[
P^{D,\psi}_tf:=\int_0^\infty P^D_sf\,\mu_t(ds),\quad t>0,\, f\in Y,
\]
is again analytic. Here $(\mu_t)_{t\ge 0}$ is a vaguely continuous semigroup of positive Borel measures on $[0,\infty)$ with $\mu_t([0,\infty))\le 1$, which represents
the Bernstein function $\psi$, i.e.
\[
e^{-t\psi(\lambda)}=\int_0^\infty e^{-\lambda s}\,\mu_t(ds),\quad t>0,\, \lambda>0.
\]

\section{Valuation of American options}
\label{sec6}

We consider  $d$-dimensional dividend paying  exponential L\'evy
models. In these models, under a risk-neutral measure (generally
nonunique), the evolution of prices, on the time interval
$[0,\infty)$, of financial assets with initial prices
$x_1>0,\dots,x_d>0$ at time $0$ is modeled by a Markov process
$\BM=(X,P_{x})$ (with $x=(x_1,\dots,x_d)$) of the form
\begin{equation}
\label{eq6.1}
X^i_t=X^i_0e^{(r-\delta_i)t+\xi^i_t},\quad t\ge0,
\end{equation}
where  $r\ge0$ is the interest rate, $\delta_i\ge0$
$i=1,\dots,d$, are dividend rates and $\xi=(\xi^1,\dots,\xi^d)$ is some L\'evy process with $\xi_0=0$ and the characteristic triple chosen so that if $\delta_i=0$, $i=1,\dots,d$, then the discounted price processes $t\mapsto e^{-rt}X^i_t=e^{\xi^i_t}$ are martingales under $P_x$.
The state space of $\BM$ is $E=\{(x_1,\dots,x_d)\in\BR^d:x_i>0,i=1,\dots,d\}$ and its life time is $\zeta=\infty$. The generator of $\BM$ has the form
\begin{equation}
\label{eq6.2}
Lf(x)=L_{BS}f(x)+L_If(x),\quad f\in C_c(\BR^d),\,x\in \BR^d,
\end{equation}
where $L_{BS}$ is the Black-Scholes operator
\[
L_{BS}=\frac12\sum^d_{i,j=1}a_{ij}x_ix_j\partial^2_{x_ix_j}  +\sum^d_{i=1}(r-\delta_i)x_i\partial_{x_i}
\]
with some positive definite symmetric matrix $a=\{a_{ij}\}$ (volatility matrix).
The operator  $L_I$ is defined by
\[
L_If(x)=\int_{\BR^d}\Big(f(x_1e^{y_1},\dots,x_de^{y_d})-f(x)
-\sum_{i=1}^dx_i(e^{y_i}-1)\partial_{x_i}f(x)\Big)\,\nu(dy),
\]
where $\nu$ is a  L\'evy measure  satisfying some additional
integrability conditions (see below). For more details on the
model see \cite{CT,KR:MF}. When $\nu=0$, the above model reduces
to the multidimensional Black-Scholes model analysed carefully in
\cite{KR:AMO,R:SSR}. From the assumptions on the model it follows
that for $i=1,\dots,d$ we have
\begin{equation}
\label{eq6.3} X^i_t=x_i+\int^t_0(r-\delta_i)X^i_s\,ds
+ M^i_t,\quad t\ge0,\quad P_{x}\mbox{-a.s.},
\end{equation}
for some martingales $M^i$, $i=1,\dots,d$ (see \cite[(2.5)]{KR:MF}).

By the definition, the value function $V_T$ of the American  option
with payoff function  $h:\BR^d\rightarrow[0,\infty)$ and
exercise time $T>0$ is given by the formula
\[
V_T(x)=\sup_{\sigma\in\TT^{T}}\mathbb E_{x}e^{-r\sigma}h(X_\sigma)=\sup_{\sigma\in\TT^{T}}\mathbb E^r_{x}h(X_\sigma),\quad x\in E,
\]
and the value function of the perpetual American  option by the formula
\[
V(x)=\sup_{\sigma\in\TT}\mathbb E_{x}e^{-r\sigma}h(X_\sigma)=\sup_{\sigma\in\TT}\mathbb E^r_{x}h(X_\sigma),\quad x\in E,
%=\sup_{\sigma\in\TT^s_{s}}E_{s,x}e^{-r(\sigma-s)}h(X_\sigma),
\]
where $\BM^r=(X,P^r_{x})$ is a Markov process which is the transformation of the process $\mathbb M$ by the multiplicative functional $A_t=e^{-rt}$, $t\ge 0$, and $\BE^r_x$ denotes the expectation with respect $P^r_x$.

In what follows, we  assume that the payoff  function $h$ is
positive, continuous and  there is $K>0$ such that
\[
|h(x)|\le K(1+|x|),\quad x\in\BR^d.
\]
As for $\nu$, we will assume that
\begin{equation}
\label{eq6.4}
\int_{\{|y|>1\}}|y|^2e^{\beta|y|}\,\nu(dy)<\infty\quad
\end{equation}
for some $\beta>1$ if $h$ is bounded, and $\beta>2$ in the
general case. Note that (\ref{eq6.4}) implies  that $E_x|X_t|^{\beta}<\infty$, $t\ge0$ (see \cite[Theorem 25.3]{Sa}). We will also assume that
\[
\det a>0.
\]
In the proof of the next theorem we  apply our general results to the above model.
In the notation of Section \ref{sec5}, in
this model $D=E$ and $\BM^r$ is the driving process. As for the data, we have $\varphi=h$ and $g=\psi=0$. Consequently, $\gamma=0$ and $w=|h|$.

\begin{theorem}
\label{th6.1}
Let $\BM$ and $h$ satisfy the assumptions described above.  Assume additionally that $\delta_i>0$ for $i=1,\dots, d$ or $\|h\|_{\infty}<\infty$. Then
$V_T(s,x)\nearrow V(x)$ as $T\rightarrow\infty$ for all $s\ge0$ and  $x\in E$. In fact, for all $T>0$ and $x\in E$,
\[
V(x)-V_T(x)\le 2e^{-rT}\|h\|_{\infty}
\]
if $h$ is bounded, and in the general case,
\begin{equation}
\label{eq6.5}
V(x)-V_T(x)\le 2K(e^{-rT}+|x|\sum^{d}_{i=1}e^{-\delta_iT}).
\end{equation}
\end{theorem}
\begin{proof}
The first inequality is immediate from Corollary \ref{cor5.1}.
To prove (\ref{eq6.5}),  we first observe that
it follows from (\ref{eq6.1}) and the fact that $t\mapsto e^{\xi^i_t}$ is a martingale that
\begin{equation}
\label{eq6.6} \mathbb E_{x}X^i_t=x_ie^{(r-\delta_i)t},\quad t\ge0.
\end{equation}
Let $\eta_t^i=e^{-rt}X^i_t$.
By (\ref{eq6.3}) and the integration by parts formula,
\begin{equation}
\label{eq6.7}
\eta_t=x_ie^{-rt}-\delta_i\int^t_0e^{-rs}X^i_{s}\,ds
+\int^t_0e^{-rs}\,dM^i_{s},\quad t\ge0,
\end{equation}
so $\eta$ is a positive supermartingale with respect to $P_x$, and hence $X$ is a supermartingale
with respect to $P^r_x$. It follows that  $h(X)$ is a supermartingale with respect to $P^r_x$. We shall show that $h(X)$
is of class (D) with respect to $P^r_x$.
Observe  that $\mathbb E_x\eta^i_t\le x_ie^{-\delta_it}\le x_i$. Therefore, by \cite[Theorem VI.6]{DM},  $\{\eta_t\}$ converges $P_x$-a.s. as $t\rightarrow\infty$ to some integrable random variable.
Moreover, under (\ref{eq6.4}),  $\lim_{p\rightarrow1^+}\mathbb E_x|X^i_1|^p=\mathbb E_x|X^i_1|=x_ie^{(r-\delta_i)}$. Hence $\mathbb E_x|\eta^i_1|^p\le x_ie^{-\delta_1/2}$ for some $p>1$, As a result, since $\eta^i$ is stationary and with independent increments, $\mathbb E_x|\eta^i_t|\le x_ie^{-\delta_it/2}$, $t\ge0$. Therefore, if $\delta_i>0$ for $i=1,\dots, d$, then
$\sup_{t\ge0}\mathbb E_x|\eta^i_t|^p<\infty$ for some $p>1$, so $\{\eta_t\}$ converges in $L^1(dP_x)$. Furthermore, by (\ref{eq6.6}), $\mathbb E_x\int^\infty_0e^{-rs}X^i_s\,ds<\infty$, so by monotone convergence,
$\int^t_0e^{-r\theta}X^i_\theta\,d\theta\rightarrow
\int^\infty_0e^{-r\theta}X^i_\theta\,d\theta$ in $L^1(dP_x)$ as $t\rightarrow\infty$.
Therefore the martingale $t\mapsto\int^t_0e^{-r\theta}\,dM_{\theta}$ is convergent in $L^1(dP_x)$. Consequently, it is uniformly integrable with respect to $P_x$ and hence
of class (D) (with respect to $P_x$) by \cite[Theorem VI.23]{DM}.  Since $t\mapsto\int^t_0e^{-rs}X^i_s\,ds$ is also of class (D) with respect to $P_x$, it follows from  (\ref{eq6.7}) that $\eta$ is of class (D)
with respect to $P_x$. Clearly, this implies that  $t\mapsto e^{-rt}h(X_t)$ is of class (D) with respect to $P_x$.
Therefore, $t\mapsto h(X_t)$ is of class (D) with respect to $P^r_x$. Applying now Corollary \ref{cor5.2} gives the second assertion of the theorem.
\end{proof}

In the
language of the option pricing theory  Theorem \ref{th6.1}  says
that in  exponential L\`evy models satisfying the assumptions
given at the beginning of this section the fair price of an
American option with payoff function $h$ and maturity $T$
converges to the  fair price of the corresponding perpetual
American option. This result generalizes the corresponding result
from \cite{R:SSR} proved (by a different method) for
multidimensional Black-Scholes  models, i.e. when $\nu=0$, and
with the additional assumption that $h$ is convex and
Lipschitz. Note, however, that the method of \cite{R:SSR} together
with the results from \cite{KR:AMO} also provides  the early
exercise formula for perpetual options.

\end{document}